%% file: main.tex
\documentclass[a4paper,reqno]{amsart}

\usepackage[a4paper,left=3cm,right=3cm,top=3cm,bottom=3cm]{geometry}

\usepackage{graphicx} 
\usepackage{tikz}
\usetikzlibrary{patterns,shapes,decorations.pathmorphing,decorations.pathreplacing,calc,arrows,cd}
\usetikzlibrary{tikzmark, fit, matrix, decorations.markings}
\usepackage{mathdots}
\usepackage{amssymb}
\usepackage{amstext}
\usepackage{amsmath}
\usepackage{amscd}
\usepackage{amsthm}
\usepackage{amsfonts}

\usepackage{array}
\usepackage{enumerate}
\usepackage{latexsym}
\usepackage{mathrsfs}

\usepackage[all,color]{xy}
\usepackage{mathtools}

\usepackage{pgfplots}
\pgfplotsset{compat=1.16}

\usepackage{booktabs}
\usepackage{multirow}
\xyoption{all}
\usepackage{pstricks}
\usepackage{comment}
\usepackage{bm}

\usepackage{todonotes}
\usepackage{tikz}

\numberwithin{equation}{section}
\newtheorem{theorem}{Theorem}[section]

\newtheorem{corollary}[theorem]{Corollary}
\newtheorem{lemma}[theorem]{Lemma}
\newtheorem*{lemma*}{Lemma}
\newtheorem*{theorem*}{Theorem}
\newtheorem{proposition}[theorem]{Proposition}
\newtheorem{definition-theorem}[theorem]{Definition-Theorem}
\newtheorem{definition-proposition}[theorem]{Definition-Proposition}

\theoremstyle{definition}
\newtheorem{definition}[theorem]{Definition}

\newtheorem{example}[theorem]{Example}

\newcommand{\RR}{\mathcal{R}}
\newcommand{\DD}{\mathcal{D}}
\newcommand{\BB}{\mathcal{B}}
\newcommand{\LL}{\mathcal{L}}
\newcommand{\XX}{\mathcal{X}}

\newcommand{\TT}{\mathfrak{T}}

\newcommand{\UU}{\mathcal{U}}

\newcommand{\YY}{\mathcal{Y}}

\newcommand{\id}{{\rm id}}
\newcommand{\mR}{\mathbb{R}}

\newcommand{\Hom}{\operatorname{Hom}\nolimits}

\newcommand{\RHom}{\mathbf{R}\strut\kern-.2em\operatorname{Hom}\nolimits}

\newcommand{\Image}{\operatorname{Im}\nolimits}
\newcommand{\coim}{\operatorname{Coim}\nolimits}

\newcommand{\im}{\operatorname{Im}\nolimits}
\newcommand{\Ex}{\operatorname{Ex}\nolimits}

\DeclareMathOperator{\omR}{\overline{\mR}}

\DeclareMathOperator{\pfdk}{\it \mathsf{Rep}^{\rm pfd}_k}
\DeclareMathOperator{\Rep}{\mathsf{Rep}}

\DeclareMathOperator{\Top}{\mathsf{Top}}
\DeclareMathOperator{\Simp}{\mathsf{Simp}}
\DeclareMathOperator{\Vect}{\mathsf{Vect}}

\usepackage{hyperref}

\newcommand{\old}[1]{{\red #1}}

\date{\today}

\begin{document}

\title{Stability of Bipath Persistence Diagrams}

\date{\today}

\author{Shunsuke Tada}
\address{Shunsuke Tada,
  Graduate School of Human Development and Environment, Kobe University, 3-11 Tsurukabuto, Nada-ku, Kobe 657-8501 Japan}
\email{205d851d@stu.kobe-u.ac.jp}


\begin{abstract}  
Recently, bipath persistent homology has been proposed as an extension of standard persistent homology, along with its visualization (bipath persistence diagram) and computational methods. 
In the setting of standard persistent homology, the stability theorem with respect to real-valued functions on a topological space is one of the fundamental results, which gives a mathematical justification for using persistent homology to noisy data. In proving the stability theorem, the algebraic stability theorem/the isometry theorem for persistence modules plays a central role.
In this point of view, the stability property for bipath persistent homology is desired for analyzing data. 
In this paper, we prove the stability theorem of bipath persistent homology with respect to \emph{bipath functions} on a topological space.
This theorem suggests a stability of bipath persistence diagrams: small changes in a bipath function (except at their ends) result in only small changes in the bipath persistence diagram.
Similar to the stability theorem of standard persistent homology, we deduce the stability theorem of bipath persistent homology by using the algebraic stability theorem/the isometry theorem of bipath persistence modules.
\end{abstract}

\maketitle
\setcounter{tocdepth}{2}
\tableofcontents

\input{1_intro}
\input{2_preliminaries}

\input{3_decomposition}

\input{4_isometry}

\section*{Acknowledgement}

The author would like to thank his Ph.D. supervisor Emerson Gaw Escolar for fruitful discussions.
He would also appreciate Toshitaka Aoki for taking the time to discuss this paper and carefully reading the draft.

This work is supported by JST SPRING, Grant Number JPMJSP2148.

\bibliographystyle{alpha} 
\bibliography{main.bib}

\end{document}

%% file: 1_intro.tex
\section{Introduction}\label{Section:intro}

Persistent homology \cite{frosini1999size},\cite{landi1997new},\cite{robins1999towards},\cite{carlsson2005computing} is one of the main tools in topological data analysis. It captures the persistence of topological features (e.g., connected components, holes, cavities, or higher dimensional holes) by a collection of intervals, called the persistence diagram. 
This tool allows us to study hidden structures in the data and has been applied in various fields including material science   \cite{hiraoka2016hierarchical}, the field of evolutionary biology  \cite{evolution}, and others
\cite{mcguirl2020topological},\cite{aktas2019persistence},\cite{belchi2018lung},\cite{PUN_PH},\cite{ESCOLAR2025101088} for example. 

A fundamental construction of a persistence diagram of a given real-valued function $f$ on a topological space $X$ is as follows: (1) Construct a sublevelset filtration $(f\leq \cdot)$ of $f$, which is given by  $(f \leq r):=\{x \in X  \mid f(x) \leq r \}$ for each $r\in \mR$. (2) Apply the $q$th homology functor (with coefficients in a field) to the filtration.  We obtain a persistent homology $V(f)$ of the function $f$, consisting of the following data: \[
    \{H_q((f\leq s)) \}_{s\in \mR} \text{, and } \{ H_q((f\leq s)) \to H_q((f\leq t))  \}_{s\leq t}.
    \] 
    (3) If $V(f)$ is a pointwise finite-dimensional (pfd for short) $\mR$-persistence module, then $V(f)$ is decomposed into the so-called \emph{interval modules} by a structure theorem for $\mR$-persistence modules \cite{crawley2015decomposition}. We collect the intervals and obtain the persistence diagram $\mathscr{B}(V(f))$ of the function $f$.
    
In the theory of persistent homology and its application to data analysis, the stability theorem \cite{Stability2007}, \cite{chazal2009proximity} for the persistent homology of a real-valued function is one of the fundamental results. It guarantees that small changes in input function $f$ result in small changes in its persistence diagram $\mathscr{B}(V(f))$, providing a mathematical justification for using persistent homology to analyze noisy data.

The algebraic stability theorem/the isometry theorem \cite{chazal2009proximity},\cite{Les2015inducedmatch},\cite{chazal2016structure},\cite{Lesnick2015interleaving} for pfd $\mR$-persistence modules plays a central role in proving the stability theorem of persistent homology. Indeed, once we have the algebraic stability theorem/the isometry theorem, we can immediately deduce the stability theorem of persistent homology.

Recently, \emph{bipath persistent homology} \cite{aoki2024bipathpersistence} has been introduced as an extension of persistent homology. It can be understood as a persistence module over finite \emph{bipath poset} $B_{n,m}:$
\begin{equation*}
 B_{n,m} \colon \ 
  \begin{tikzcd}[row sep=0.1em,column sep = 1.4em, inner sep=0pt]
    & 1 \rar[] & 2 \rar[] & \cdots \rar[] & n \ar[dr] & \\
    {-\infty} \ar[ur] \ar[dr] & & & &  & {+\infty}. \\
    & 1' \rar[] & 2' \rar[] & \cdots \rar[] & m' \ar[ur] &
  \end{tikzcd}
\end{equation*}
This allows us to study the persistence of topological features across a \emph{bipath filtration}, a pair of filtrations connected at their ends to compare the two filtrations, as shown below:
\begin{equation*}
  S\colon \ 
  \begin{tikzcd}[row sep=0.1em,column sep = 1.4em, inner sep=0pt]
    & S_1 \rar[hookrightarrow] & S_2 \rar[hookrightarrow] & \cdots \rar[hookrightarrow] & S_n \ar[dr,hookrightarrow] & \\
    S_{-\infty} \ar[ur,hookrightarrow] \ar[dr,hookrightarrow] & & & &  & S_{+\infty}. \\
    & S_{1'} \rar[hookrightarrow] & S_{2'} \rar[hookrightarrow] & \cdots \rar[hookrightarrow] & S_{m'} \ar[ur,hookrightarrow] &
  \end{tikzcd}
\end{equation*}

One important property of bipath persistent homology is that it decomposes into interval modules \cite{aoki2023summand}, \cite{aoki2024bipathpersistence}. This property allows for intuitive visualization of the persistence of topological features across bipath filtrations by \emph{bipath persistence diagrams} \cite{aoki2024bipathpersistence}, see also Subsection~\ref{subsec:finitebipath}. In this research direction, computational methods for bipath persistence diagrams have been proposed in \cite{aoki2024bipathpersistence}, \cite[Section 5]{asashiba2024IntMultp}, \cite[Section 4]{alonso2024Bipath}.
Moreover, from a theoretical perspective, fibered arc codes \cite{alonso2024Bipath} for two-parameter persistent homology have been considered and compared with the fibered barcode \cite{lesnick2015interactive} and the interval
rank invariant \cite{Asa2024IntReplacement}.

For an application of bipath persistent homology in data analysis, a stability property of bipath persistent homology is required.

In this paper, we show the stability of bipath persistent homology.
To be more precise, we set the (continuous) bipath poset $B$ (Definition~\ref{def:bipathposet}), where the underlying set is given by $(\bigsqcup_{i=1,2}\mR\times \{i\}) \sqcup \{\pm \infty\}$, displayed below:
\begin{equation*}
\begin{tikzpicture}
    \coordinate (b) at (2,0);
    \coordinate (minf) at (1,-0.5); 
     \coordinate (maxf) at (9,-0.5); 
    \coordinate (d) at ($(8,0)$);
    \draw[line width = 0.3mm] ($(b)$)--($(d)$); 
        \draw[line width = 0.3mm] ($(b)-(0,1)$)--($(d) -(0,1)$); 
   \node at ($0.5*(b)+0.5*(d)$) [above]{$\mathbb{R}\times \{1\}$}; 
      \node at ($0.5*(b)+0.5*(d)-(0,1)$) [above]{$\mathbb{R}\times \{2\}$}; 
  \node at ($(minf)$) [left]{$-\infty$}; 
  \node at ($(minf)-(1,0)$) [left]{$B \colon$}; 
   \node at ($(minf)$) [right]{$\bullet$}; 
  \node at ($(maxf)$) [right]{$+\infty.$}; 
   \node at ($(maxf)$) [left]{$\bullet$}; 
\end{tikzpicture}    
\end{equation*}
For a topological space $X$, a \emph{bipath function} $f=(f_1,f_2)$ is a pair of $B$-valued maps on $X$ such that $\Image(f_i) \subseteq (\mR\times \{i\}) \cup \{\pm \infty\}$ ($i=1,2$) with $f_1^{-1}(\{-\infty\}) =f_2^{-1}(\{-\infty\})$.
We construct the bipath persistence diagram of a bipath function using a procedure similar to that of persistent homology, as explained above.
\begin{enumerate}
    \item Construct a \emph{bipath sublevelset filtration} $(f\leq \cdot)$ of the bipath function $f$ by
\begin{equation*}
       (f\leq b) \coloneqq
        \begin{cases}
         f_1^{-1}(\{- \infty\}) & \text{if } b= - \infty,
         \\
         X & \text{if } b= + \infty,
         \\
        \{x \in X \ | \  f_i(x) \leq r \} & \text{if $b =(r,i) \in \mathbb{R}\times \{i\}$} \text{ for $i=1,2$}. 
        \end{cases}
\end{equation*}
    \item Apply the $q$th homology functor (with some coefficient in a field) to the bipath sublevelset filtration. We obtain a $B$-persistence module (\emph{bipath persistence module}) $V(f)$ which we call \emph{bipath persistent homology} of $f$.
    If $V(f)$ is a pfd bipath persistence module, then we say that $f$ is \emph{tame}.
    \item  If $f$ is tame, by structure theorem (Theorem \ref{thm:decomp}), $V(f)$ is decomposed into interval modules. We collect the interval modules and obtain the bipath persistence diagram $\mathscr{B}(V(f))$ of the bipath function $f$. 
\end{enumerate}

As an analogue of the stability theorem of the standard persistent homology, we show the stability theorem of bipath persistent homology (Theorem \ref{thm:stability}) with respect to bipath functions.
 The theorem suggests that small changes in a bipath function $f$ (except at its ends) result in small changes in its bipath persistence diagram $\mathscr{B}(V(f))$. 
This result suggests a mathematical justification for the use of bipath persistent homology to analyze noisy data.

The stability theorem is followed by the algebraic stability theorem /the isometry theorem for bipath persistence modules (Theorem \ref{thm:isometry}).

\vskip.2\baselineskip
\noindent
{\bf{Related work.}}

Recently, Alonso and Liu obtained a stability result on $B_{n,m}$-persistence modules \cite[Theorem 3.5]{alonso2024Bipath}, which is derived from the algebraic stability theorem of the zigzag-persistence modules (\cite{Les2018AlgebraicStability} \cite[Theorem 3.2]{bjerkevik2021onthesatbility}).


The difference of Theorem \ref{thm:isometry} and 
\cite[Theorem 3.5]{alonso2024Bipath} arises from interleaving distances, which are defined by using translations: 
    In this paper, we use translations on the bipath poset $B$ given by \eqref{eq:superfamilyB}, whereas they use translations on a subset of $\mR^{{\rm op}}\times \mR$ to apply the algebraic stability theorem of zigzag-persistence modules. Thus, these interleaving distances are different and, therefore, as-is incomparable.

This work is partially based on the author's Ph.D. thesis.

%% file: 2_preliminaries.tex
\section{Preliminaries}

In this section, we give the basics of persistence modules and the distances between persistence modules.

Throughout, let $k$ be a field. We denote by $\Vect_k$ the category of $k$-vector spaces.

\subsection{Persistence modules}\label{subsec:persistence}
Let $P=(P,\leq )$ be a (not necessarily finite) partially ordered set (poset for short). 
     We regard $P$ as a category whose objects are elements of $P$ and there exists a unique morphism $p \to q$ if and only if $p \leq q$.

    We denote by $\Rep_k(P)$ the functor category from $P$ to $\Vect_k$. A \emph{persistence module} over $P$ (or, $P$-persistence module) is an object in $\Rep_k(P)$.  
    For a $P$-persistence module $V$, we denote by $V_p$ the vector space associated with $p \in P$, and by $V(p,q)$ the $k$-linear map associated with $p\leq q$.  
    A morphism $\alpha \colon V\to W$ between $P$-persistence modules is a natural transformation. That is, it is a family of $k$-linear maps $\alpha=(\alpha_p\colon V_p \to W_p )_{p\in P}$ such that $ \alpha_q \circ {V}(p,q) = W (p,q) \circ \alpha_p$ for any $p\leq q$. 
    We denote by $\Hom_P(V,W)$ the set of morphisms from $V$ to $W$.

For posets $P$ and $Q$, a map $h \colon P \to Q$ is called an order-preserving map if $p\leq q$ in $P$ implies $h(p) \leq h(q)$ in $Q$. An order-preserving map is called an isomorphism if it is a bijection and its inverse is also order-preserving.
For an order-preserving map $h \colon P \to Q$, we have a functor $h^* \colon \Rep_k(Q) \to \Rep_k(P)$ given as follows:
\begin{itemize}
    \item For $V\in \Rep(Q)$, a $P$-persistence module $h^*(V)$ is given by $h^*(V)_p := V_{h(p)}$  and $h^*(V)(p,q)  := V(h(p),h(q))\colon V_{h(p)} \to V_{h(q)}$ for each $p\leq q$ in $P$.
    \item For $\alpha \colon V \to W$ in $\Rep(Q)$, a morphism $h^*(\alpha)  \colon h^*(V) \to h^*(W)$ is given by $h^*(\alpha)_p:=\alpha_{ h(p)} \colon V_{h(p)} \to W_{h(p)}$ for each $p\in P$.
\end{itemize}

    The category of \emph{pointwise finite-dimensional $P$-persistence modules} (pfd $P$-persistence modules for short), denoted by $\pfdk(P)$, is the full subcategory of $\Rep_k(P)$ consisting of all $P$-persistence modules $V$ such that $V_p$ is a finite-dimensional $k$-vector space for each $p \in P$.  
    In this setting, we have the following result, which asserts that pfd persistence modules are uniquely decomposed into indecomposable pfd persistence modules.

    \begin{theorem}[cf. {\cite[Theorem 1.1]{botnan2020decomposition}, \cite[Theorem 1 (ii)]{Azumaya_1950}}]
        Let $P$ be a poset. Any pfd $P$-persistence module is uniquely decomposed into indecomposable pfd $P$-persistence modules up to isomorphism and permutations of terms.
    \end{theorem}
 Thanks to the above theorem, any pfd $P$-persistence module $V$ is the form $V\cong \bigoplus_{\lambda \in \Lambda } V_\lambda$, where each $V_\lambda$ is a pfd indecomposable $P$-persistence module. We denote by $\mathscr{B}(V)$ the multiset $ \{V_{\lambda}\}_{\lambda \in \Lambda}$, which is uniquely determined up to isomorphism.

    In this paper, intervals and interval modules are central objects. 

\begin{definition} 
We say that a subset $P'$ of $P$ is an \emph{interval} if it satisfies the following two conditions:
\begin{enumerate}
    \item \emph{Convexity}: For $p,q \in P'$ and $r \in P$, if $p \leq r \leq q$, then $r \in P'$.
    \item \emph{Connectivity}: If $p,q \in P'$, there exists a sequence $p = z_0, z_1, \hdots , z_{\ell} = q$ of elements of $P'$ such that $z_{i-1}$ and $z_i$ are comparable for any $i \in \{ 1,\hdots,\ell \}$. 
\end{enumerate}
 We denote by $\mathbb{I}(P)$ the set of all intervals of $P$. 
\end{definition}
Notice that any convex set is a disjoint union of intervals.

 \begin{definition}
     Let $P$ be a poset.
     \begin{itemize}
         \item [{\rm (1)}]
For a non-empty subset $A$ of $P$, we define an \emph{upset} $A^\uparrow := \{ 
p \in P \mid \exists a \in A  \text{ such that } \ a \leq p \}$ and a \emph{downset} $A^\downarrow := \{ 
p \in P \mid \exists a \in A \text{ such that } \ p \leq a \}$, which are clearly convex in $P$. In addition, we set $\emptyset^{\uparrow}:= P$ and $ \emptyset^{\downarrow}: = P$.
\item  [{\rm (2)}] For $p,q\in P$, we have an interval $[p,q]:=\{r\in P \mid p\leq r\text{ and } r\leq q \}$ called a \emph{segment} of $P$.
     \end{itemize}
 \end{definition}

We have a well-known characterization of convex sets.

\begin{lemma}\label{lem:convex}
Let $C$ be a subset of $P$. Then, $C$ is convex if and only if $C^\uparrow \cap C^\downarrow = C$ holds.
\end{lemma}

We will recall the definition of interval modules.

\begin{definition}
Let $P$ be a poset and $I$ be an interval of $P$. We define the pfd $P$-persistence module $k_I$ by
  \begin{equation}
        (k_I)_p :=
        \begin{cases}
        k & \text{if $p\in I$}, \\
        0 & \text{otherwise,}
        \end{cases}\quad \quad
        {k_I}(p,q) :=
        \begin{cases}
        \id_k & \text{if $p, q\in I$ with $p\leq q$}, \\
        0 & \text{else.}
        \end{cases}
    \end{equation}

An \emph{interval module} over $P$ is a $P$-persistence module isomorphic to $k_I$ for an interval $I$ of $P$.
   
\end{definition} 
Any interval module is an indecomposable $P$-persistence module \cite[Corollary 5.6]{blanchette2021homological}.
 We say that a pfd $P$-persistence module $V$ is \emph{interval-decomposable} if $V$ is decomposed into interval modules. 
Any interval-decomposable $P$-persistence module $V$ is of the form $V\cong \bigoplus_{I\in \mathbb{I}(P)} k_I^{m_V(I)}$, where $m_V(I)$ is a non-negative integer for each $I\in \mathbb{I}(P)$. In this case, we identify 
 \[ \mathscr{B}(V) =\{k_I \text{ with the multiplicity $m_V(I)$} \}\]
 with the multiset of intervals
 \[
 \{I\in \mathbb{I}(P) \text{ with the multiplicity $m_V(I)$}\},
 \]
if there is no risk of confusion.

For two intervals $I$ and $J$ of $P$, let $\Omega(I,J)$ be the set of all connected components $C$ of $I \cap J$ satisfying
\begin{equation}\label{eq:condition}
    I \cap C^\downarrow \subseteq C \text{ and } 
    J\cap C^\uparrow \subseteq C. 
\end{equation}

\begin{lemma}[{\cite[Proposition 5.5]{blanchette2021homological}}]\label{lem:well-def}
Let $I$ and $J$ be intervals of $P$. For each connected component $C$ of $I\cap J$ in  $\Omega(I,J)$, we have a morphism $\phi_C \colon k_I \to k_J$ given by 
\begin{equation*}
    (\phi_C)_p := 
    \begin{cases}
          \id_k  & \text{ if } p \in C, \\
           0 & \text{ else.}
    \end{cases}
\end{equation*}
Moreover, the set $\{\phi_C \mid C \in \Omega(I,J)  \}$ forms a basis of $\Hom_{P} (k_I,k_J)$.  In particular, $\Omega(I,J)=\emptyset$ if and only if $\Hom_{P} (k_I,k_J) =0$.
\end{lemma}

Finally, we give the following property of intervals.
\begin{lemma}\label{lem:convex_pullback}
Let $P$ and $Q$ be posets, and $h\colon P\to Q$ be an order-preserving map. For an interval $I$ of $Q$, we have the following:  
\begin{enumerate}
    \item [{\rm (1)}] $h^{-1}(I)$ is a convex set of $P$.
    \item [{\rm (2)}] $h^*(k_I)$ is isomorphic to $\bigoplus_{C}k_{C}$, where the direct sum ranges over all connected components $C$ of $h^{-1}(I)$. In particular, it is interval-decomposable.
\end{enumerate}
\end{lemma}

\subsection{Distances}\label{subsec:distances}
In this subsection, we define and compare interleaving, bottleneck interleaving, and bottleneck distances.

\subsubsection{Translations}
We refer the readers to \cite{desilva2018theoryinterleavingscategoriesflow} for a general theory of interleaving.
Let $P$ be a poset.
 A \emph{translation} on $P$ is an order-preserving map $h\colon P \to P$ such that  $p\leq h(p)$ for any $p \in P$.  
We denote by $\text{Trans}(P)$ the set of all translations on $P$. For the identity map $\id_P$ and the composition $\circ$ of maps, the triplet $(\text{Trans} (P), \id_P, \circ )$ is naturally equipped with a monoid structure. Each homomorphism of monoid 
\[
\Lambda \colon  (\mR_{\geq 0},0, +) \to (\text{Trans} (P), \id_P, \circ )
\]
defines a monoid action of $\mR_{\geq0}$ on $P$ provided by $\mR_{\geq0}\times P \to P; (r,p)\mapsto \Lambda_\epsilon(p)$ for all $\epsilon \in \mR_{\geq0}$ and $p\in P$.
We call such $\Lambda$ an \emph{$\mR_{\geq0}$-action} on $P$.
Notice that giving an $\mR_{\geq0}$-action $\Lambda$ on $P$ is equivalent to giving a family of translations $\{\Lambda_\epsilon\}_{\epsilon\geq0}$ of $P$ satisfying the following properties:
\begin{itemize}\label{eq:conditionLambda}
    \item  For all $\epsilon \geq 0$ and $p \in P$, we have $p \leq \Lambda_\epsilon(p)$.
    \item  $\Lambda_0 = \id_P$.
    \item For all $\epsilon, \zeta \geq 0$, we have  $\Lambda_\epsilon \circ \Lambda_\zeta = \Lambda_{\epsilon + \zeta}$.
\end{itemize}
 In addition, we call such $\Lambda$ an \emph{$\mR$-action} on $P$ if $\Lambda_\epsilon$ is an automorphism of posets for each $\epsilon \geq0$. In this case, we write $\Lambda_{-\epsilon}$ for the inverse map $\Lambda_\epsilon^{-1}$. 
For example, we have an $\mR$-action on the poset $P=\mR$ (with respect to the usual order) given by 
\begin{equation}\label{eq:action_on_R}
\Lambda_\epsilon(r):=r+\epsilon \text{ ($r \in \mR$).}    
\end{equation}

Let $\Lambda$ be an $\mR_{\geq0}$-action on $P$.
For each $\epsilon \geq 0$, we call the induced functor   
 $\Lambda_\epsilon^* \colon \Rep(P) \to \Rep(P)$ a \emph{$\Lambda_\epsilon$-shift functor}. 
 For a $P$-persistence module $V$, we write $V(\epsilon)$ for $\Lambda_\epsilon^*(V)$. In addition, for a morphism $f\colon V \to W$ of $P$-persistence modules, we write $f(\epsilon)$ for $\Lambda_\epsilon^*(f)$.
 Since $\Lambda_\epsilon$ is a translation, we have a morphism $V \to V(\epsilon)$ provided by $ V(p,\Lambda_\epsilon(p))\colon V_p \to V(\epsilon)_p$ for each $p\in P$. We denote this morphism by $V_{0\to \epsilon}$.

With respect to an $\mR_{\geq0}$-action on $P$, we will define the following distances.

\begin{definition}[$\Lambda_\epsilon$-interleaving and the interleaving distance]\label{def:d_I} 
Let $V$ and $W$ be pfd $P$-persistence modules. For $\epsilon\geq0$, a \emph{$\Lambda_\epsilon$-interleaving} between $V$ and $W$ is a pair of morphisms $\alpha \colon V\to W(\epsilon)$ and $\beta \colon W\to V(\epsilon)$ such that $\beta (\epsilon)\circ \alpha  =  V_{0\to 2\epsilon}$ and $ \alpha(\epsilon)\circ \beta = W_{0 \to2\epsilon}$.
If there exists a $\Lambda_\epsilon$-interleaving between $V$ and $W$, we say that $V$ and $W$ are \emph{$\Lambda_\epsilon$-interleaved}. The \emph{interleaving distance} between $V$ and $W$ is defined by  \begin{equation*}
     d_{\rm{I}}^{\Lambda} (V, W):= \inf \{\epsilon \in \mR_{\geq0} \ | \  \text{$V$ and $W$ are $\Lambda_\epsilon$-interleaved}\}.
 \end{equation*}   
\end{definition}

Next, we will define a bottleneck interleaving distance between pfd $P$-persistence modules.
A \emph{partial matching} between  multisets $\mathscr{A}$ and $\mathscr{B}$ is a bijection $\sigma \colon \mathscr{A}'\to \mathscr{B}'$ for some subsets $\mathscr{A}' \subseteq \mathscr{A}$ and $\mathscr{B}' \subseteq \mathscr{B}$, and  we write it by $\sigma \colon \mathscr{A} \nrightarrow  \mathscr{B}$.
 In this case, $\coim \sigma$ and $\im \sigma$ denote $\mathscr{A}'$ and $\mathscr{B}'$, respectively.

We say that a pfd $P$-persistence module $V$ is \emph{$\Lambda_{\epsilon}$-trivial} if $V_{0\to \epsilon} = 0$. Otherwise, we say that $V$ is \emph{$\Lambda_{\epsilon}$-significant}.  Notice that $V$ is $\Lambda_{2\epsilon}$-trivial if and only if $V$ and $0$ are $\Lambda_\epsilon$-interleaved. 

\begin{definition}[bottleneck $\Lambda_\epsilon$-interleaving and the bottleneck interleaving distance]\label{def:bottle2}
Let $V$ and $W$ be pfd $P$-persistence modules. For  $\epsilon\geq0$, a \emph{bottleneck $\Lambda_\epsilon$-interleaving} between $V$ and $W$ is a partial matching $\sigma \colon \mathscr{B}(V) \nrightarrow \mathscr{B}(W)$ satisfying the following properties$\colon$
    \begin{enumerate}
        \item For all $M \in \mathscr{B}(V) \setminus \coim \sigma$, the pfd $P$-persistence module $M$ is $\Lambda_{2\epsilon}$-trivial. 
        \item For all $N \in \mathscr{B}(W) \setminus \im \sigma$, the pfd $P$-persistence module $N$ is $\Lambda_{2\epsilon}$-trivial. 
        \item For all $M \in \coim \sigma$, the pfd $P$-persistence modules $M$ and $\sigma(M)$ are $\Lambda_\epsilon$-interleaved.
    \end{enumerate}
  We say that $V$ and $W$ are \emph{bottleneck $\Lambda_\epsilon$-interleaved} if there exists a bottleneck $\Lambda_\epsilon$-interleaving between $V$ and $W$. Note that this property does not depend on a choice of $\mathscr{B}(V)$ and $\mathscr{B}(W)$. The \emph{bottleneck interleaving distance} between $V$ and $W$
 is defined by 
 \begin{equation*}
     d_{\rm{BI}}^\Lambda (V, W):= \inf \{\epsilon \in\mR_{\geq0} \ | \  \text{$V$ and $W$ are bottleneck $\Lambda_\epsilon$-matched}\}. 
 \end{equation*}    
\end{definition}

 We have the following observation by definitions of bottleneck $\Lambda_\epsilon$-interleaving and $\Lambda_\epsilon$-interleaving.

\begin{proposition}[{\cite[p.12]{oudot2024differential}}]\label{prop:dI<dB}
    Let $V$ and $W$ be pfd $P$-persistence modules and $\epsilon\geq0$. If there exists a bottleneck $\Lambda_\epsilon$-interleaving between $V$ and $W$, then there exist a $\Lambda_\epsilon$-interleaving between $V$ and $W$. Therefore, we have the following inequality:
    \begin{equation*}
    d_{\rm{I}}^\Lambda(V,W) \leq d_{\rm{BI}}^\Lambda(V,W).
    \end{equation*}
\end{proposition}

Next, we define a bottleneck distance between multisets of intervals of $P$. 
For $\epsilon\geq 0$ and $A\subseteq P$, we define the \emph{$\Lambda_\epsilon$-thickening} of $A$ by
\begin{equation*}
    \Ex^\Lambda_{\epsilon}(A) :=  \Lambda_\epsilon^{-1}(A)^\uparrow \cap \Lambda_\epsilon(A)^\downarrow.
\end{equation*}


\begin{proposition}
Suppose that $\Lambda$ is an $\mR$-action on $P$.  
For an interval $I$ of $P$, we have $I \subseteq \Ex_\epsilon^\Lambda(I)$ and $\Ex_\epsilon^\Lambda(I)\in \mathbb{I}(P)$ for all $\epsilon\geq 0$.
\end{proposition}

\begin{proof}
    Let $I$ be an interval of $P$. Then, $\Ex^\Lambda_{\epsilon}(I) =  \Lambda_{-\epsilon}(I^\uparrow) \cap \Lambda_{\epsilon}(I)^\downarrow$ holds by using assumption on $\Lambda$. It is shown in \cite[p.57 and Proposition A.4]{clause2024gen} that $\Lambda_{-\epsilon}(I^\uparrow) \cap \Lambda_{\epsilon}(I)^\downarrow$ is an interval including $I$. 
\end{proof}





The bottleneck distance is defined by using $\Lambda_\epsilon$-thickening of intervals.
\begin{definition}[$\Lambda_\epsilon$-matching and the bottleneck distance]\label{def:matching} 
Let $\mathscr{A}$ and $\mathscr{B}$ be multisets of intervals of $P$. For $\epsilon\geq0$, a \emph{$\Lambda_\epsilon$-matching} between $\mathscr{A}$ and $\mathscr{B}$ is a partial matching $\sigma \colon \mathscr{A} \nrightarrow \mathscr{B}$ satisfying the following properties.
\begin{enumerate}
    \item For all $I \in \mathscr{A} \setminus \coim \sigma$, the interval module $k_I$ is $\Lambda_{2\epsilon}$-trivial.
    \item For all $J \in \mathscr{B} \setminus \im \sigma$, the interval module $k_J$ is $\Lambda_{2\epsilon}$-trivial.
    \item If $\sigma(I) = J$ then $I \subseteq \Ex^\Lambda_{\epsilon}(J)$ and $J \subseteq \Ex^\Lambda_{\epsilon}(I)$.
\end{enumerate}
We say that $\mathscr{A}$ and $\mathscr{B}$ are \emph{$\Lambda_\epsilon$-matched} if there exists a $\Lambda_\epsilon$-matching between $\mathscr{A}$ and $\mathscr{B}$.  The \emph{bottleneck distance} between  $\mathscr{A}$ and $\mathscr{B}$ is defined by
    \begin{equation*}
d_{\rm{B}}^\Lambda(\mathscr{A},\mathscr{B}):=  \inf \{ \epsilon \in \mR_{ \geq 0}  \ | \ \text{$\mathscr{A}$ and $\mathscr{B}$ are $\Lambda_\epsilon$-matched}\}.
    \end{equation*}
\end{definition}

\subsubsection{Additional observations of the bottleneck interleaving and the bottleneck distances}
In this subsubsection, we fix a poset $P$ and an $\mR$-action $\Lambda$ on $P$. In this case, recall that $\Lambda_\epsilon$ is an automorphism of $P$ for any $\epsilon\geq0$. 
We study a relation between the bottleneck interleaving distance and the bottleneck distance with respect to the $\mR$-action $\Lambda$. 

\begin{proposition}\label{prop:Raction_condition}
    Let $P$ be a poset and $\Lambda$ be an $\mR$-action on $P$. For any interval $I$ of $P$ and $\epsilon \geq 0$,  if  $p \in I^{\uparrow}$  and  $\Lambda_\epsilon(p) \in \Lambda_\epsilon(I)^\downarrow$, then $p \in I$ holds.
\end{proposition}

\begin{proof}
    Let $I$ be an interval of $P$. We assume $p \in I^\uparrow$ and $\Lambda_\epsilon (p) \in \Lambda_\epsilon(I)^\downarrow$.
By $p \in I^\uparrow$, there exists $a \in I$ such that $a\leq p$. On the other hand, by $\Lambda_\epsilon (p)\in \Lambda_\epsilon(I)^\downarrow$, there exists $b \in I$ such that $\Lambda_\epsilon (p) \leq \Lambda_\epsilon(b)$. Since $\Lambda_\epsilon$ is an automorphism, we have $p\leq b$. Since $I$ is convex and $a\leq p\leq b$ with $a,b\in I$, we have $p\in I$. 
\end{proof}

\begin{lemma}\label{lem:support_interval}
Let $I$ be an interval of $P$. Then, $\Lambda_{-\epsilon}(I)$ is an interval of $P$ for any $\epsilon\geq0$. Therefore, $k_I(\epsilon) = k_{\Lambda_{-\epsilon}(I)}$ is an interval module. 
\end{lemma}

\begin{proof}
The assertion follows from a fact that automorphisms of posets preserve both convexity and connectivity of subsets.
\end{proof}

\begin{lemma}\label{lem:meet-condition} 
    Let $I$ and $J$ are intervals of $P$ such that $I \subseteq \Ex^\Lambda_{\epsilon}(J)$ and $J \subseteq \Ex^\Lambda_{\epsilon}(I)$ for an $\epsilon\geq0$. Then, the set $\Omega(I,\Lambda_{-\epsilon}(J))$ (resp. $ \Omega(J,\Lambda_{-\epsilon}(I))$) coincides with the set of connected components of $I \cap \Lambda_{-\epsilon}(J)$ (resp. $J \cap \Lambda_{-\epsilon}(I)$). 
\end{lemma}

\begin{proof}Let $I$ and $J$ be intervals of $P$.
 We show that $\Omega(I \cap \Lambda_{-\epsilon}(J))$ coincides with the set of connected components of $I \cap \Lambda_{-\epsilon}(J)$.
 The assertion for $ \Omega(J,\Lambda_{-\epsilon}(I))$ can be obtained by a similar discussion.
 
 Let $C$ be a connected component of $I\cap \Lambda_{-\epsilon}(J)$. 
    We show that $C$ satisfies \eqref{eq:condition}, meaning that both  $I\cap C^{\downarrow} \subseteq C$ and $ \Lambda_{-\epsilon}(J) \cap C^{\uparrow} \subseteq C$ hold. We first show the former inclusion relation. 
    Let $A:=I\cap C^{\downarrow}$.    For any $a \in A$, there exists $c \in C$ such that $a \leq c$. In addition, by $c \in C \subseteq  I \cap \Lambda_{-\epsilon}(J)$, we have $a \in \Lambda_{-\epsilon}(J)^\downarrow $. 
    Moreover, by $a \in A \subseteq  I \subseteq \Ex^\Lambda_{\epsilon}(J) \subseteq  \Lambda_{-\epsilon}(J)^\uparrow$, we have $a \in \Lambda_{-\epsilon}(J)^\downarrow  \cap  \Lambda_{-\epsilon}(J)^\uparrow = \Lambda_{-\epsilon}(J)$, where the last equality is followed by Lemmas~\ref{lem:convex} and  \ref{lem:support_interval}. By the above discussion, we have $a \in I \cap \Lambda_{-\epsilon}(J)$. 
    Since $C$ is a connected component of $I \cap \Lambda_{-\epsilon}(J)$, and since we have $a \leq c$, the element $a$ belongs to the same component as $c$. Thus, we have $a \in C$ and obtain the desired inclusion relation $A \subseteq C$. 
    
    Next, we show $B := \Lambda_{-\epsilon}(J) \cap C^{\uparrow}  \subseteq C$. For any $b \in B$, there exists $c \in C$ such that $c \leq b$. In addition, by $c\in C \subseteq I \cap \Lambda_{-\epsilon}(J)$, we have $b \in  I^\uparrow$. 
    On the other hand, we have $\Lambda_\epsilon (b) \in J \subseteq \Ex^\Lambda_{\epsilon}(I) \subseteq {\Lambda_\epsilon(I)}^\downarrow $.  By Proposition \ref{prop:Raction_condition}, we have $b \in I$. Thus, we obtain $b \in  I \cap \Lambda_{-\epsilon}(J)$. By $c \leq b$, the element $b$ is in the same component as $c$. 
    Thus, we have $b \in B$ and obtain the desired inclusion relation $B \subseteq C$.  By the above discussion, we conclude that the set $\Omega(I,\Lambda_{-\epsilon}(J))$  coincides with the set of connected components of $I \cap \Lambda_{-\epsilon}(J)$.
 This completes the proof.
\end{proof}

\begin{lemma}\label{lem:capcap}
    Let $I$ and $J$ be intervals of $P$ such that $I \subseteq \Ex^\Lambda_{\epsilon}(J)$ and $J \subseteq \Ex^\Lambda_{\epsilon}(I)$ for an $\epsilon\geq0$. Then we have 
\begin{equation}\label{eq:two_eq}
       I \cap \Lambda_{-\epsilon}(J) \cap \Lambda_{-2\epsilon}(I) = I \cap \Lambda_{-2\epsilon}(I) \text{ and } J \cap \Lambda_{-\epsilon}(I) \cap \Lambda_{-2\epsilon}(J) = J \cap \Lambda_{-2\epsilon}(J).     
    \end{equation}
\end{lemma}
\begin{proof}
Let $I$ and $J$ be intervals of $P$.
    We first show the former equality of \eqref{eq:two_eq}.
    For any $p \in I \cap \Lambda_{-2\epsilon}(I)$, we show $p \in \Lambda_{-\epsilon}(J)$.
    By assumption, we have $p \in I \subseteq \Ex^\Lambda_{\epsilon}(J) \subseteq \Lambda_{-\epsilon}(J)^\uparrow$.  Thus, there exists $x \in \Lambda_{-\epsilon}(J)$ such that $x \leq p$. Therefore, we have $J \ni \Lambda_\epsilon(x) \leq \Lambda_\epsilon(p)$, which implies $\Lambda_\epsilon(p) \in J^{\uparrow}$. 
    On the other hand, since $p \in \Lambda_{-2\epsilon}(I)$, we have $\Lambda_\epsilon(\Lambda_\epsilon(p)) = \Lambda_{2\epsilon}(p) \in I \subseteq \Ex^\Lambda_{\epsilon}(J) \subseteq \Lambda_\epsilon(J)^{\downarrow}$.  By Proposition \ref{prop:Raction_condition}, we have $\Lambda_\epsilon(p) \in J$. Thus, we obtain $p \in \Lambda_{-\epsilon}(J)$.  By the above discussion, we obtain the former equality of \eqref{eq:two_eq}.

    By changing the role of $I$ and $J$ in the above discussion, we obtain the latter equality of \eqref{eq:two_eq}. 
\end{proof}

\begin{proposition}\label{prop:matching-induce-matching} Let $P$ be a poset with an $\mR$-action $\Lambda$. Let $I$ and $J$ be intervals of $P$ such that $I \subseteq \Ex^\Lambda_{\epsilon}(J)$ and $J \subseteq \Ex^\Lambda_{\epsilon}(I)$ for an $\epsilon\geq0$.
Then, $k_I$ and $k_J$ are $\Lambda_\epsilon$-interleaved.
\end{proposition}

\begin{proof}
For intervals $I$ and $J$ of $P$,
    by Lemmas~\ref{lem:well-def} and \ref{lem:meet-condition}, we have morphisms $\alpha \colon k_I \to k_J(\epsilon)$ and $\beta \colon k_J \to k_I(\epsilon)$ such that 
    \begin{equation*}
        \alpha_p := 
        \begin{cases}
        \id_k & \text{ if } p \in I \cap \Lambda_{-\epsilon}(J),\\
        0 & \text{ else}, 
        \end{cases}
        \text{ and }
           \beta_p := 
        \begin{cases}
        \id_k & \text{ if } p \in J \cap \Lambda_{-\epsilon}(I),\\
        0 & \text{ else}, 
        \end{cases}
    \end{equation*}
for any $p \in P$.
Then, the pair $\alpha$ and $\beta$ gives a $\Lambda_\epsilon$-interleaving between $k_I$ and $k_J$. Indeed, for any $p\in P$, we have 
    \begin{eqnarray*}
        (\beta(\epsilon)\circ \alpha)_p  &=& 
        \begin{cases}
        \id_k & \text{ if } p \in I \cap \Lambda_{-\epsilon}(J)\cap \Lambda_{-2\epsilon}(I), \\
        0 & \text{ else}, 
        \end{cases} \\
         &\overset{ \text{Lemma \ref{lem:capcap}}}{=}& 
        \begin{cases}
        \id_k & \text{ if } p \in I \cap \Lambda_{-2\epsilon}(I) ,\\
        0 & \text{ else}, 
        \end{cases} \\
        &=& ((k_I)_{0\to2\epsilon})_p.
    \end{eqnarray*}
Thus, we have $\beta(\epsilon)\circ \alpha = (k_I)_{0\to 2\epsilon}$. Similarly, we have $\alpha(\epsilon)\circ \beta = (k_J)_{0 \to 2\epsilon}$. Therefore, the pair $\alpha$ and $\beta$ is a $\Lambda_\epsilon$-interleaving between $k_I$ and $k_J$. 
\end{proof}

As a consequence, we relate $\epsilon$-matching and bottleneck $\epsilon$-interleaving.

\begin{proposition}\label{lem:inequality}
Let $P$ be a poset with an $\mR$-action $\Lambda$. For interval-decomposable pfd $P$-persistence modules $V$ and $W$,
     if there exists a $\Lambda_\epsilon$-matching between $\mathscr{B}(V)$ and $\mathscr{B}(W)$ for an $\epsilon\geq0$, then it gives a bottleneck $\Lambda_\epsilon$-interleaving between $V$ and $W$. Therefore, we have
    \begin{equation*}
        d_{\rm{BI}}^\Lambda(V,W) \leq d_{\rm{B}}^\Lambda(\mathscr{B}(V),\mathscr{B}(W)).
    \end{equation*}
\end{proposition}
\begin{proof}
Let $V$ and $W$ be interval-decomposable pfd $P$-persistence modules.
If there exists a $\Lambda_\epsilon$-matching between $\mathscr{B}(V)$ and $\mathscr{B}(W)$, then by Proposition \ref{prop:matching-induce-matching}, it gives a bottleneck $\Lambda_\epsilon$-interleaving between $V$ and $W$. 
\end{proof}

\subsubsection{Relation to graph matching}\label{sec:graph}
  The bottleneck $\epsilon$-matching between pfd persistence modules can be expressed using graph theory. The following discussion is essentially based on  \cite{bjerkevik2021onthesatbility}, where it gives an upper bound on a bottleneck interleaving distance between rectangle decomposable pfd $\mR^n$-persistence modules using an interleaving distance between them.

We recall the basic terminology of graph theory. Let $G=(V, E)$ be a graph, where $V$ is a (not necessarily finite) set whose elements are called \emph{vertices}, and $E$ is a (not necessarily finite) set of unordered pairs $(v_1,v_2)$ of vertices, whose elements are called \emph{edges}. 
A \emph{matching} in $G$ is a set of edges such that no two edges share a common vertex.
For subsets $V' \subseteq V$ and $E' \subseteq E$, we say that $E'$ \emph{covers} $V'$ if any vertex in $V'$ is adjacent to an edge in $E'$. 
The \emph{full subgraph} induced by $V'$ is the subgraph whose vertex set is $V'$ and whose edge set consists of all the edges $(v_1,v_2)\in E$ such that $ v_1, v_2 \in V'$. 
The set of \emph{neighbors} of a vertex $v$ in $G$, denoted by $N_G(v)$, is the set of vertices that are adjacent to $v$.
 For a subset $V' \subseteq V$, we write $N_G(V')$ for $\bigcup_{v \in V'} N_G(v)$.

A \emph{bipartite graph} is a graph whose vertex set is a disjoint union of two sets, $X$ and $Y$, and whose edge set $E$ consists of edges connecting a vertex in $X$ to a vertex in $Y$. In this case, we write this bipartite graph by $G=(X, Y; E)$. In addition, for subsets $X' \subseteq X$ and $Y'\subseteq Y$, we denote by $G(X', Y')$ the full subgraph of $G$ induced by $X'$ and $Y'$.

To state Hall's marriage theorem, we set the following conditions.
\begin{definition}\label{condition}
 Let $G =(X, Y; E)$ be a bipartite graph. We say that $G$ satisfies Conditions {\rm (H)} and {\rm (H$'$)} respectively if it satisfies the following.
    \begin{itemize}
        \item[{\rm (H)}] For any finite subset $X'$ of $X$, we have  $|X'| \leq | N_G (X')| $.
    \end{itemize}\label{conditionH}
    \begin{itemize}
        \item[{\rm (H$'$)}] For any finite subset $Y'$ of $Y$, we have  $|Y'| \leq | N_G (Y')|$.
    \end{itemize}\label{conditionH'}
\end{definition}

\begin{theorem}[{\cite[Theorem 1]{Hall1935}}]
    Let $G=(X,Y;E)$ be a bipartite graph 
    such that $N_G(x)$ is finite for all $x\in X$.  The following statements are equivalent.
    \begin{enumerate}
        \item[\rm (a)] The graph $G$ satisfies Condition {\rm (H)}.
        \item[\rm (b)]There exists a matching in $G$ that covers $X$.
    \end{enumerate}
\end{theorem}

Then, we have the following observation (see \cite[p.111]{bjerkevik2021onthesatbility} for its proof).
\begin{corollary}\label{cor:matching}
    Let $G=(X, Y; E)$ be a bipartite graph. Let $X'$ and $Y'$ be subsets of $X$ and $Y$ respectively such that both neighbors $N_{G}(x')$ and $N_{G}(y')$ are finite for any $x' \in X'$ and $y' \in Y'$.
    If the full subgraphs $G(X', Y)$ and $G(X, Y')$ satisfy Conditions {\rm (H)} and {\rm (H$'$)} respectively, then there exists a matching in $G$ that covers $X'\sqcup Y'$. 
\end{corollary}


Let $P$ be a poset and $\Lambda$ be an $\mR_{\geq0}$-action on $P$.
We study a bottleneck $\Lambda_\epsilon$-interleaving below. We first construct a bipartite graph using two pfd $P$-persistence modules $V$ and $W$. For $M \in \mathscr{B}(V)$ and $\epsilon \geq 0$, we set \begin{equation}\label{eq:mu}
    \mu_\epsilon^\Lambda(M) := \{N \in  \mathscr{B}(W)  \mid  \text{$M$ and $N$ are $\Lambda_\epsilon$-interleaved}\}. 
\end{equation}
In addition, for $\mathscr{A} \subseteq \mathscr{B}(V)$, we set 
\begin{equation}
\mu_\epsilon^\Lambda(\mathscr{A}) := \bigcup_{M\in \mathscr{A}} \mu_\epsilon^\Lambda(M).
\end{equation}
Using the above notation, we define a bipartite graph $(\mathscr{B}(V), \mathscr{B}(W); E_{\mu_\epsilon^\Lambda})$ where $E_{\mu_\epsilon^\Lambda}$ is the set of edges given by $\bigcup_{M\in \mathscr{B}(V) }\{ (M,N)  \mid  N \in \mu_\epsilon^\Lambda(M)  \}$. Notice that $(\mathscr{B}(V), \mathscr{B}(W); E_{\mu_\epsilon^\Lambda})$ and $(\mathscr{B}(W), \mathscr{B}(V); E_{\mu_\epsilon^\Lambda})$ are the same bipartite graph.

We denote by $\mathscr{B}_{2\epsilon}(V)$ the multiset of $\Lambda_{2\epsilon}$-significant $P$-persistence modules in $\mathscr{B}(V)$.  We give a connection between bottleneck $\Lambda_\epsilon$-interleaving and matching in a bipartite graph. 


%
\begin{proposition}[{\cite[p.4]{bjerkevik2021onthesatbility}}]\label{prop:matching-and-ematching}
    Let $P$ be a poset with an $\mR_{\geq0}$-action $\Lambda$.
    Let $V$ and $W$ be pfd $P$-persistence modules. For any 
 $\epsilon\geq0$, the following statements are equivalent.
    \begin{enumerate}
        \item[\rm (a)] There exists a bottleneck $\Lambda_\epsilon$-interleaving between $V$ and $W$.
        \item[\rm (b)] There exists a matching in the bipartite graph $(\mathscr{B}(V), \mathscr{B}(W); E_{\mu_\epsilon^\Lambda})$ that covers $\mathscr{B}_{2\epsilon}(V)\sqcup \mathscr{B}_{2\epsilon}(W)$.
    \end{enumerate}
\end{proposition}

\begin{lemma}[{\cite[p.110]{bjerkevik2021onthesatbility}}]\label{lem:finness} 
 Let $V$ and $W$ be pfd $P$-persistence modules and $G = (\mathscr{B}(V), \mathscr{B}(W); E_{\mu_\epsilon^\Lambda})$ be the bipartite graph for an $\epsilon\geq0$.  
 If $M \in \mathscr{B}_{2\epsilon}(V)$, then the neighbors $N_{G}(M)$ of $M$ is finite. 
\end{lemma}

 By Lemma \ref{lem:finness},  Corollary \ref{cor:matching}, and Proposition \ref{prop:matching-and-ematching}, we obtain the following observation. It gives a sufficient condition for the existence of a bottleneck $\Lambda_\epsilon$-interleaving between two pfd persistence modules.

\begin{proposition}[{\cite[p.111]{bjerkevik2021onthesatbility}}]\label{thm:matching}
   Let $P$ be a poset with an $\mR_{\geq0}$-action $\Lambda$.  Let $V$ and $W$ be pfd $P$-persistence modules and $\epsilon\geq0$. For the bipartite graph $G =(\mathscr{B}(V), \mathscr{B}(W); E_{\mu_\epsilon^\Lambda})$, if the full subgraphs $G(\mathscr{B}_{2\epsilon}(V), \mathscr{B}(W))$ and $G(\mathscr{B}(V), \mathscr{B}_{2\epsilon}(W))$ satisfy Conditions {\rm (H)} and  {\rm (H$'$)}, respectively, then there exists a bottleneck $\Lambda_\epsilon$-interleaving between $V$ and $W$. 
\end{proposition}

%% file: 3_decomposition.tex
\section{Bipath persistence}\label{sec:bipath}

In this section, we first recall finite bipath posets, introduced in \cite{aoki2024bipathpersistence} in the context of the interval-decomposability of persistence modules over finite posets. We also recall bipath persistent homology, which was proposed in \cite{aoki2024bipathpersistence} as an extension of standard one-parameter persistent homology.
Next, we introduce the continuous version of finite bipath posets (Definition \ref{def:bipathposet}). Pfd persistence modules over this poset also admit interval-decomposability (Theorem \ref{thm:decomp}).
Then, we introduce bipath functions (Definition \ref{def:bipath_function}), which induce bipath persistent homology.

\subsection{Finite bipath posets}\label{subsec:finitebipath}
In this subsection, we recall the basics of finite bipath posets from \cite{aoki2024bipathpersistence}.


\begin{definition}[{\cite[Definition 2.3]{aoki2024bipathpersistence}}]\label{def:finite_bipathposet}
Let $n$ and $m$ be non-negative integers. The finite \emph{bipath poset} $B_{n,m}$ is a poset whose underlying set is $\{1,\hdots, n \} \sqcup \{1',\hdots, m'\} \sqcup \{\pm \infty\}$ with the order given by $\{-\infty \leq 1 \leq \cdots \leq  n \leq +\infty \}$ and $\{-\infty\leq 1' \leq \cdots \leq  m' \leq +\infty\}$.  We display the poset $B_{n,m}$ by the following Hasse diagram: 
\begin{equation}
 B_{n,m} \colon \ 
  \begin{tikzcd}[row sep=0.1em,column sep = 1.4em, inner sep=0pt]
    & 1 \rar[] & 2 \rar[] & \cdots \rar[] & n \ar[dr] & \\
    {-\infty} \ar[ur] \ar[dr] & & & &  & {+\infty}. \\
    & 1' \rar[] & 2' \rar[] & \cdots \rar[] & m' \ar[ur] &
  \end{tikzcd}
\end{equation}
 \end{definition}

One important property of pfd persistence modules over finite bipath posets is that they are decomposed into interval modules, analogous to the case of the $A$-type poset (i.e., posets whose Hasse diagrams are Dynkin diagrams of type $A$).
In fact, we have the following result on the interval-decomposability of pfd persistence modules over finite posets.

\begin{theorem}[{\cite[Theorem 1.3]{aoki2023summand}}]\label{thm:finite_bipath}
    For a finite connected poset $P$, the following conditions are
equivalent.
\begin{enumerate}
    \item [{\rm (a)}] Any pfd $P$-persistence module is interval-decomposable.
    \item [{\rm (b)}]The poset $P$ is either an $A$-type poset or a finite bipath poset.
\end{enumerate}
\end{theorem}
For any pfd persistence module $V$ over a finite bipath poset, we can regard $\mathscr{B}(V)$ as a multiset of intervals by its interval-decomposability. We call $\mathscr{B}(V)$ the \emph{bipath persistence diagram} of $V$. 
Similar to the standard persistence diagrams, we can visualize bipath persistence diagrams in the plane as follows. 
 Let $U:=\{-\infty \leq 1 \leq \cdots \leq  n \leq +\infty \}$ and $D:=\{-\infty\leq 1' \leq \cdots \leq  m' \leq +\infty\}$. Then, we divide $\mathbb{I}(B_{n,m})$ into five pairwise disjoint sets $\mathbb{I}(B_{n,m}) =\UU(B_{n,m}) \sqcup \DD(B_{n,m}) \sqcup \BB(B_{n,m}) \sqcup \LL(B_{n,m}) \sqcup \RR(B_{n,m})$ as follows: 
\begin{itemize}
\item $\UU(B_{n,m}) := \{[s,t] \mid s\leq t \in [1,n] \}$.
\item $\DD(B_{n,m}):=\{ [t,s] \mid t\leq s \in [1',m']  \}$.
\item $\BB(B_{n,m}) :=\{B_{n,m}\}$.
    \item $\LL(B_{n,m}) :=\{[-\infty, t] \cup [-\infty, s] \mid t \in U  \text{ and } s \in  D \text{ with } s, t \neq +\infty \}$.
    \item $\RR(B_{n,m}):= \{[s,+\infty] \cup [ t, +\infty] \mid s \in U\text{ and } t \in  D \text{ with } s, t \neq -\infty \}$.
\end{itemize}
If an interval $I$
in $\mathbb{I}(B_{n,m})\setminus\BB(B_{n,m})$ is given by one of the above forms using $s$ and $t$, then we write $I =\langle s, t\rangle$, see Table \ref{tab:intervalsB}. Using the notation, any interval can be visualized by a point in the plane, see Figure~\ref{fig:bipathPD}.

\begin{table}[h]
\renewcommand{\arraystretch}{1.5}
    \begin{tabular}{cccccccc}
        $\UU(B_{n,m})$ & \hspace{-2mm}$\DD(B_{n,m})$ & \hspace{-2mm}$\BB(B_{n,m})$ & \hspace{-2mm} $\LL(B_{n,m})$& \hspace{-2mm}$\RR(B_{n,m})$ 
        \\ 
 \begin{tikzpicture}[baseline = 0mm, scale =0.6]
        \coordinate (x) at (0.4,0);
        \coordinate (y) at (0,0.6); 
        \coordinate (0) at ($0*(x) + 0*(y)$);
        \coordinate (1) at ($1*(x) + 1*(y) - 0.1*(x)$);
        \coordinate (n) at ($5*(x) + 1*(y)$);
        \coordinate (11) at ($6*(x) + 0*(y) - 0.1*(x)$);
        \coordinate (1') at ($1*(x) + -1*(y) - 0.1*(x)$);
        \coordinate (m') at ($5*(x) + -1*(y)$);
        \draw (0)--(1)--(n)--(11)--(m')--(1')--cycle;
        
        \coordinate (s) at ($2*(x) + 1*(y)$);
        \coordinate (t) at ($4*(x) + 1*(y)$);
        \fill (s)circle(0.7mm)node[above]{$s$};
        \fill (t)circle(0.7mm)node[above]{$t$};
        \draw[line width = 0.7mm] (s)--(t);
        \node[left] at (0) {\footnotesize$-\infty$};
        \node[right] at (11) {\footnotesize$+\infty$}; 
        \node[above] at (1) {\footnotesize$1$}; 
        \node[above] at (n) {\footnotesize$n$}; 
        \node[below] at (1') {\footnotesize$1'$}; 
        \node[below] at (m') {\footnotesize$m'$}; 
        \end{tikzpicture} 
        & 
        \hspace{-4mm}
         \begin{tikzpicture}[baseline = 0mm, scale =0.6]
        \coordinate (x) at (0.4,0);
        \coordinate (y) at (0,0.6); 
        \coordinate (0) at ($0*(x) + 0*(y)$);
        \coordinate (1) at ($1*(x) + 1*(y) - 0.1*(x)$);
        \coordinate (n) at ($5*(x) + 1*(y)$);
        \coordinate (11) at ($6*(x) + 0*(y) - 0.1*(x)$);
        \coordinate (1') at ($1*(x) + -1*(y) - 0.1*(x)$);
        \coordinate (m') at ($5*(x) + -1*(y)$);
        \draw (0)--(1)--(n)--(11)--(m')--(1')--cycle;
        
        \coordinate (s) at ($4*(x) + -1*(y)$);
        \coordinate (t) at ($2*(x) + -1*(y)$);
        \fill (s)circle(0.7mm)node[above]{$s$};
        \fill (t)circle(0.7mm)node[above]{$t$};
        \draw[line width = 0.7mm] (t)--(s);
        \node[left] at (0) {\footnotesize$-\infty$};
        \node[right] at (11) {\footnotesize$+\infty$}; 
        \node[above] at (1) {\footnotesize$1$}; 
        \node[above] at (n) {\footnotesize$n$}; 
        \node[below] at (1') {\footnotesize$1'$}; 
        \node[below] at (m') {\footnotesize$m'$}; 
        
        \end{tikzpicture} 
        & 
        \hspace{-4mm}
         \begin{tikzpicture}[baseline = 0mm, scale = 0.6]
        \coordinate (x) at (0.4,0);
        \coordinate (y) at (0,0.6); 
        \coordinate (0) at ($0*(x) + 0*(y)$);
        \coordinate (1) at ($1*(x) + 1*(y) - 0.1*(x)$);
        \coordinate (n) at ($5*(x) + 1*(y)$);
        \coordinate (11) at ($6*(x) + 0*(y) - 0.1*(x)$);
        \coordinate (1') at ($1*(x) + -1*(y) - 0.1*(x)$);
        \coordinate (m') at ($5*(x) + -1*(y)$);
        \draw (0)--(1)--(n)--(11)--(m')--(1')--cycle;
        \node[left] at (0) {\footnotesize$-\infty$};
        \node[right] at (11) {\footnotesize$+\infty$}; 
        \node[above] at (1) {\footnotesize$1$}; 
        \node[above] at (n) {\footnotesize$n$}; 
        \node[below] at (1') {\footnotesize$1'$}; 
        \node[below] at (m') {\footnotesize$m'$}; 
        \draw[line width = 0.7mm] (0)--(1)--(n)--(11)--(m')--(1')--cycle;
        \end{tikzpicture} 
        & \hspace{-4mm}
          \begin{tikzpicture}[baseline = 0mm, scale =0.6]
        \coordinate (x) at (0.4,0);
        \coordinate (y) at (0,0.6); 
        \coordinate (0) at ($0*(x) + 0*(y)$);
        \coordinate (1) at ($1*(x) + 1*(y) - 0.1*(x)$);
        \coordinate (n) at ($5*(x) + 1*(y)$);
        \coordinate (11) at ($6*(x) + 0*(y) - 0.1*(x)$);
        \coordinate (1') at ($1*(x) + -1*(y) - 0.1*(x)$);
        \coordinate (m') at ($5*(x) + -1*(y)$);
        \draw (0)--(1)--(n)--(11)--(m')--(1')--cycle;
        
        \coordinate (s) at ($3*(x) + -1*(y)$);
        \coordinate (t) at ($3*(x) + 1*(y)$);
        \fill (s)circle(0.7mm)node[above]{$s$};
        \fill (t)circle(0.7mm)node[above]{$t$};
        \draw[line width = 0.7mm] (s)--(1')--(0)--(1)--(t);
        \node[left] at (0) {\footnotesize$-\infty$};
        \node[right] at (11) {\footnotesize$+\infty$}; 
        \node[above] at (1) {\footnotesize$1$}; 
        \node[above] at (n) {\footnotesize$n$}; 
        \node[below] at (1') {\footnotesize$1'$}; 
        \node[below] at (m') {\footnotesize$m'$}; 
        \end{tikzpicture} 
        & 
        \hspace{-4mm}
            \begin{tikzpicture}[baseline = 0mm, scale =0.6]
        \coordinate (x) at (0.4,0);
        \coordinate (y) at (0,0.6); 
        \coordinate (0) at ($0*(x) + 0*(y)$);
        \coordinate (1) at ($1*(x) + 1*(y) - 0.1*(x)$);
        \coordinate (n) at ($5*(x) + 1*(y)$);
        \coordinate (11) at ($6*(x) + 0*(y) - 0.1*(x)$);
        \coordinate (1') at ($1*(x) + -1*(y) - 0.1*(x)$);
        \coordinate (m') at ($5*(x) + -1*(y)$);
        \draw (0)--(1)--(n)--(11)--(m')--(1')--cycle;
        
        \coordinate (s) at ($3*(x) + 1*(y)$);
        \coordinate (t) at ($3*(x) + -1*(y)$);
        \fill (s)circle(0.7mm)node[above]{$s$};
        \fill (t)circle(0.7mm)node[above]{$t$};
        \draw[line width = 0.7mm] (s)--(n)--(11)--(m')--(t);
        \node[left] at (0) {\footnotesize$-\infty$};
        \node[right] at (11) {\footnotesize$+\infty$}; 
        \node[above] at (1) {\footnotesize$1$}; 
        \node[above] at (n) {\footnotesize$n$}; 
        \node[below] at (1') {\footnotesize$1'$}; 
        \node[below] at (m') {\footnotesize$m'$}; 
        
        \end{tikzpicture} 
    \end{tabular}
    \caption{Intervals of $B_{n,m}$. 
    Each hexagon represents the Hasse diagram of the bipath poset $B_{n,m}$. 
    The middle one is $B_{n,m}$ itself. For each $\XX \in \{\LL(B_{n,m}),\RR(B_{n,m}),\UU(B_{n,m}),\DD(B_{n,m})\}$, 
    we represent intervals $I=\langle s,t \rangle \in \XX$ by thick lines.  
    }
    \label{tab:intervalsB}
\end{table}

\begin{figure}[htbp]
\begin{center}
\includegraphics[width=110mm]{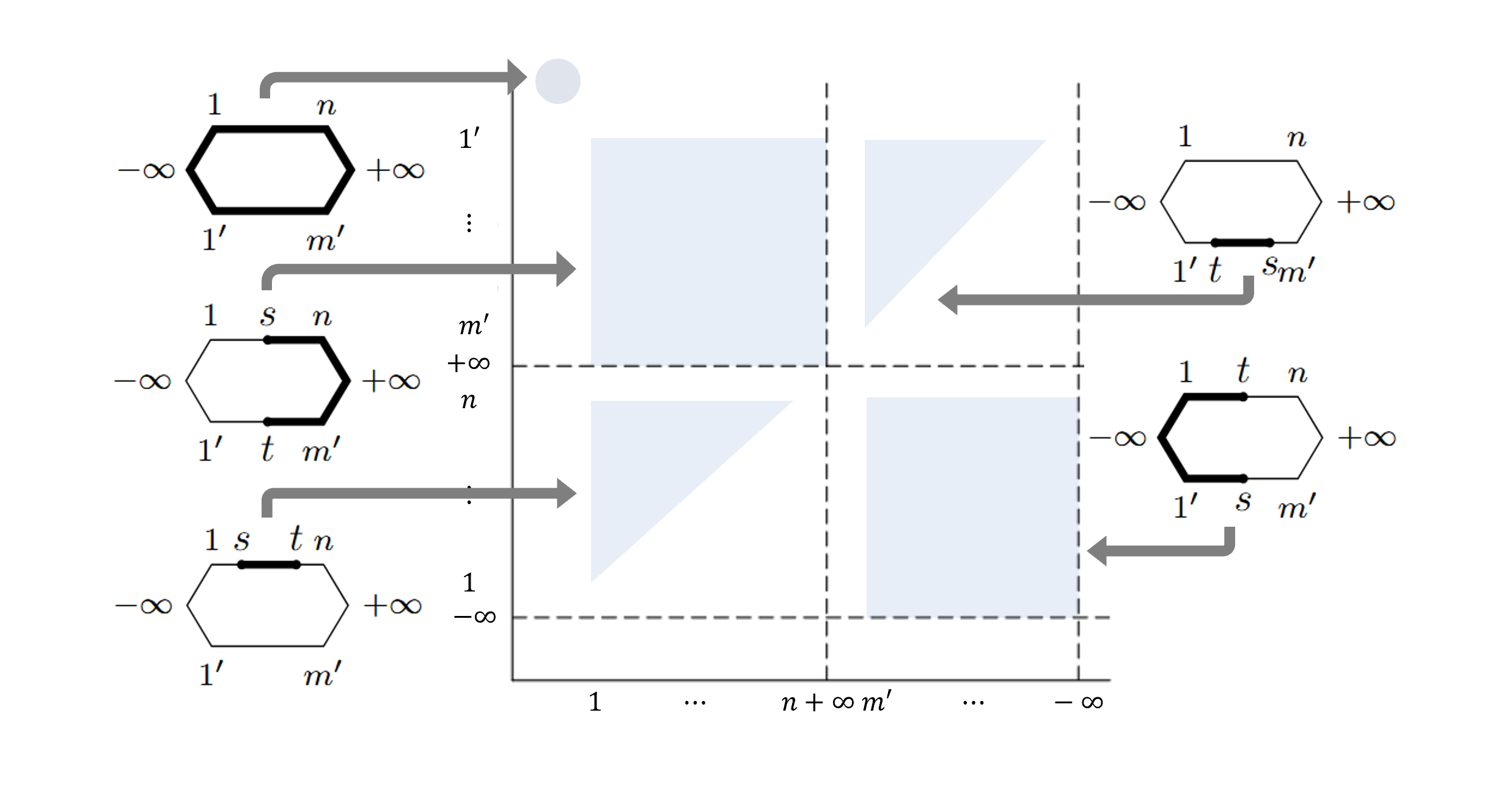}
\caption{A correspondence of intervals of $B_{n,m}$ and points in the plane.
For an interval $I=\langle s,t \rangle \in \mathbb{I}(B_{n,m})\setminus\BB(B_{n,m})$, we plot a point at $(s,t)$. 
For the interval $B_{n,m}$, we plot a point at the upper left region in the plane. 
}
\label{fig:bipathPD}
\end{center}
\end{figure}

We next consider bipath persistent homology.
For this, we consider a $B_{n,m}$-filtration $S$ (a functor from $B_{n,m}$ to the category of topological spaces/simplicial complexes, denoted by $\Top$ and $\Simp$, such that $S_a \hookrightarrow S_b$ for $a\leq b\in B_{n,m}$), which is displayed by the following diagram:
\begin{equation*}
  S\colon \ 
  \begin{tikzcd}[row sep=0.1em,column sep = 1.4em, inner sep=0pt]
    & S_1 \rar[hookrightarrow] & S_2 \rar[hookrightarrow] & \cdots \rar[hookrightarrow] & S_n \ar[dr,hookrightarrow] & \\
    S_{-\infty} \ar[ur,hookrightarrow] \ar[dr,hookrightarrow] & & & &  & S_{+\infty}. \\
    & S_{1'} \rar[hookrightarrow] & S_{2'} \rar[hookrightarrow] & \cdots \rar[hookrightarrow] & S_{m'} \ar[ur,hookrightarrow] &
  \end{tikzcd}
\end{equation*}
By applying the $q$th homology functor $H_q(- ; k)$ to $S$, we obtain the bipath persistent homology of $S$.  
\begin{equation*}
H_q(S;k)\colon \ 
  \begin{tikzcd}[row sep=0.1em,column sep = 1.4em, inner sep=0pt]
    & H_q(S_1;k) \rar[hookrightarrow] & H_q(S_2;k) \rar[hookrightarrow] & \cdots \rar[hookrightarrow] & H_q(S_n;k) \ar[dr,hookrightarrow] & \\
   H_q( S_{-\infty};k) \ar[ur,hookrightarrow] \ar[dr,hookrightarrow] & & & &  & H_q(S_{+\infty};k). \\
    & H_q(S_{1'};k )\rar[hookrightarrow] & H_q(S_{2'};k) \rar[hookrightarrow] & \cdots \rar[hookrightarrow] & H_q(S_{m'};k) \ar[ur,hookrightarrow] &
  \end{tikzcd}
\end{equation*}

If $H_q(S;k)$ is a pfd $B_{n,m}$-persistence module, then it decomposes into interval modules, and we obtain the bipath persistence diagram $\mathscr{B}(H_q(S;k))$.

\begin{example}
Let $\Delta$ be an abstract simplicial complex whose $j$th faces are given by 
\[
\Delta^0:= \{ \{a\},\{b\},\{c\}\}, \Delta^1:= \{ \{a,b\},\{b,c\}, \{a,c \} \}.
\]
Let $S$ be a $B_{2,2}$-filtration: 
\begin{equation*}
  S\colon \ 
  \begin{tikzcd}[row sep=0.7em,column sep = 4.0em, inner sep=0pt]
    & S_1 \rar[hookrightarrow] & S_2  \ar[dr,hookrightarrow] & \\
    S_{-\infty} \ar[ur,hookrightarrow] \ar[dr,hookrightarrow] &   &  & S_{+\infty}, \\
    & S_{1'} \rar[hookrightarrow] & S_{2'} \rar[ru,hookrightarrow]  &
  \end{tikzcd}
\end{equation*}
where each $S_b \ (b\in B_{2,2})$ is given by 
\begin{eqnarray*}
    S_{-\infty} &:=& \{\{a\},\{b\}\},\\ \  S_{+\infty} &:=& \Delta^1 , \\
        S_{1} &:=& \Delta^0 , \\          S_{1'} &:=&  \Delta^0 ,  \\
        S_{2} &:=& \Delta^1 , \\     
 S_{2'} &:=&  \Delta^0\cup \{\{a,b\} \}. 
\end{eqnarray*}
See Figure \ref{fig1} for the geometric realization of the $B_{2,2}$-filtration of simplicial complexes. 

\begin{figure}[htbp]
\begin{center}
\includegraphics[width=90mm]{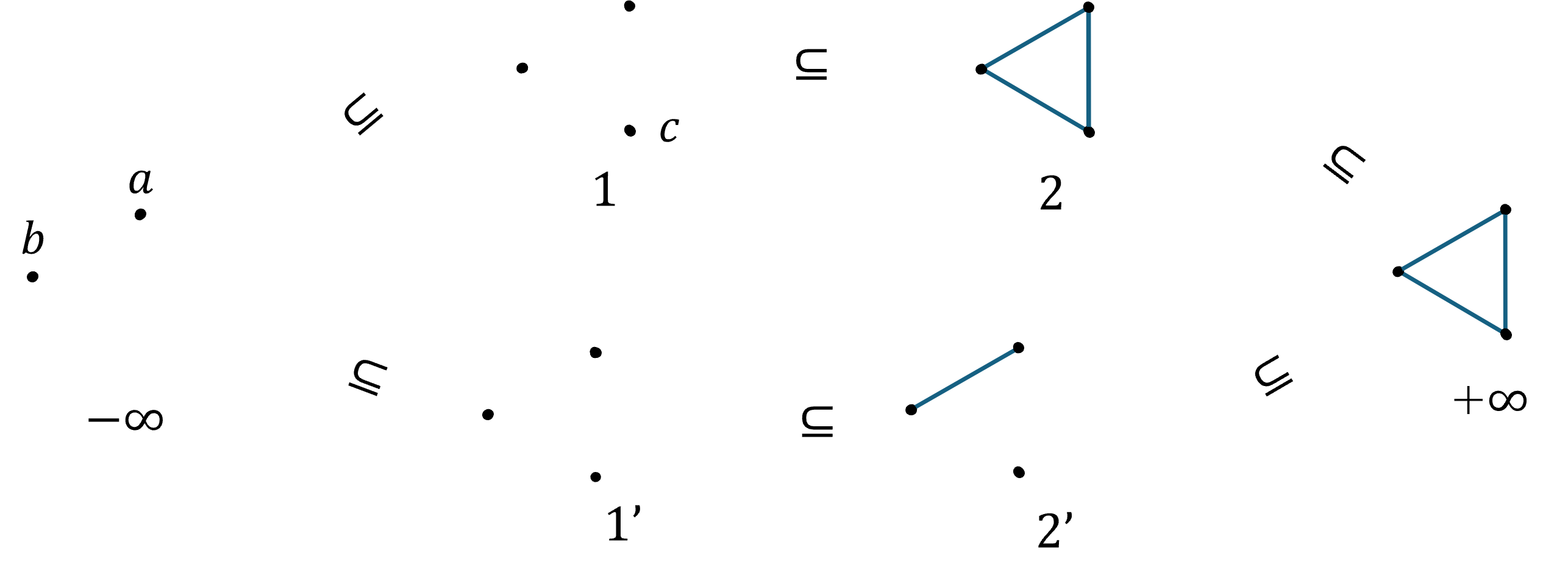}
\caption{A geometric realization of the $B_{2,2}$-filtration $S$.
}\label{fig1}
\end{center}
\end{figure}

We apply the $0$th and $1$st homology functors to $S$ and obtain 
\[H_0(S;k) \cong k_{B_{2,2}} \oplus k_{\langle 1,1\rangle } \oplus k_{\langle 2',1'\rangle } \oplus  k_{\langle 1',1\rangle } \text{ and } \ H_1(S;k) \cong  k_{\langle 2,+\infty\rangle }, \]
and their persistence diagrams
\[\mathscr{B}(H_0(S;k)) =\{ B_{2,2}, \langle 1,1\rangle , \langle 2',1'\rangle, \langle 1',1\rangle  \} \text{ and } \mathscr{B}(H_1(S;k))= \{\langle 2,+\infty\rangle\}.
 \]
See Figure \ref{fig2} for their visualization.

\begin{figure}[htbp]
\begin{center}
\includegraphics[width=60mm]{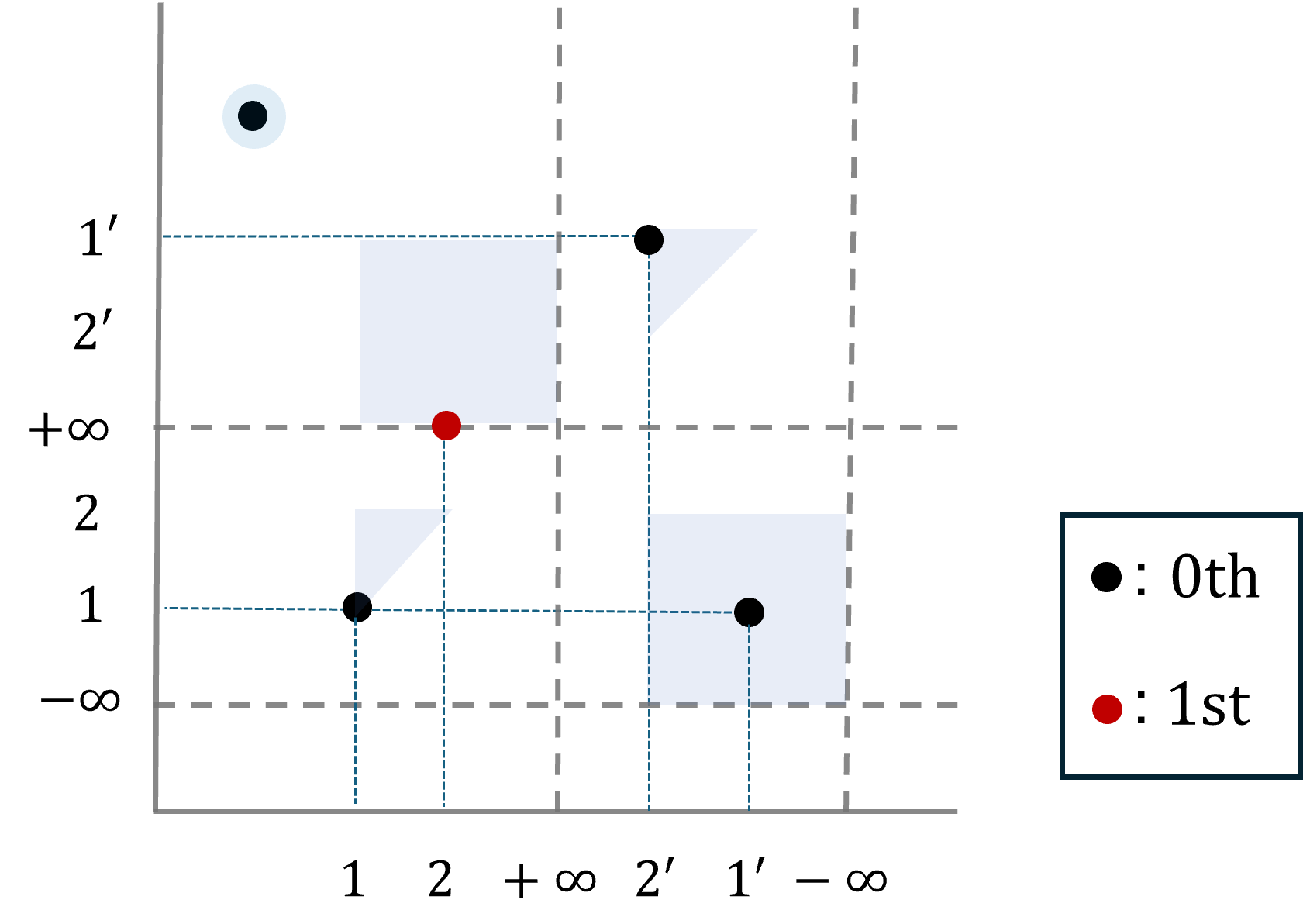}
\caption{A visualization of the $0$th and $1$st biapth persistence diagrams of the $B_{2,2}$-filtration $S$.}
\label{fig2}
\end{center}
\end{figure}

\end{example}


\subsection{The bipath poset}\label{subsec:TheBipathPoset}
In this subsection, we introduce a continuous version of finite bipath posets (Definition \ref{def:bipathposet}). In addition, we introduce bipath functions (Definition \ref{def:bipathfunction}) which induce bipath persistent homology. 

\begin{definition}\label{def:bipathposet}
The \emph{bipath poset} is a poset whose underlying set is $(\mR \times \{1,2\}) \sqcup \{\pm \infty \}$ and its order is given by  
\[x \leq y \colon \Longleftrightarrow \begin{cases}
        x = -\infty, \text{or} \\
        y = + \infty, \text{or} \\
        x=(s,i) \leq (t,i)=y \text{ in } \mR\times \{i\} \text{ with } s \leq t \text{ and $i\in \{1,2\}.$}  
    \end{cases}
    \] 
We denote by $B$ the bipath poset.
 \end{definition}
The bipath poset $B$ is displayed as follows.
\begin{equation*}
\begin{tikzpicture}

    \coordinate (b) at (2,0);
    \coordinate (minf) at (1,-0.5); 
     \coordinate (maxf) at (9,-0.5); 
    \coordinate (d) at ($(8,0)$);
    \draw[line width = 0.3mm] ($(b)$)--($(d)$); 
        \draw[line width = 0.3mm] ($(b)-(0,1)$)--($(d) -(0,1)$); 
   \node at ($0.5*(b)+0.5*(d)$) [above]{$\mathbb{R}\times \{1\}$}; 
      \node at ($0.5*(b)+0.5*(d)-(0,1)$) [above]{$\mathbb{R}\times \{2\}$}; 
  \node at ($(minf)$) [left]{$-\infty$}; 
  \node at ($(minf)-(1,0)$) [left]{$B \colon$}; 
   \node at ($(minf)$) [right]{$\bullet$}; 
  \node at ($(maxf)$) [right]{$+\infty$.}; 
   \node at ($(maxf)$) [left]{$\bullet$}; 
\end{tikzpicture}    
\end{equation*}

We call $B$-persistence modules \emph{bipath persistence modules}. Any pfd bipath persistence module admits interval-decomposability similar to the pfd persistence modules over finite bipath posets.

\begin{theorem}\label{thm:decomp}
  Any pfd bipath persistence module is uniquely decomposed into a direct sum of interval modules up to isomorphism.
\end{theorem}

One can prove Theorem \ref{thm:decomp} using a similar argument given in \cite[Section 4]{aoki2024bipathpersistence}, which shows the interval-decomposability of pfd persistence modules over finite bipath posets. While its proof is given in the setting of finite bipath posets, the idea of approach (the so-called matrix problems methods) almost equally applies to our setting. We leave the proof to the reader.

 Thanks to the interval-decomposability of pfd bipath persistence modules $V$, we can regard $\mathscr{B}(V)$ as a multiset of intervals. We also call $\mathscr{B}(V)$ the \emph{bipath persistence diagram} of $V$. 

We next describe intervals of $B$.
First, we divide $\mathbb{I}(B)$ into the following five pairwise disjoint sets of intervals similar to the case of finite bipath posets as follows:
\begin{itemize}
     \item $\UU\coloneq \{ I \in \mathbb{I}(B) \mid I \subseteq \mR\times \{1\} \}$.
     \item $\DD \coloneq \{ I \in \mathbb{I}(B) \mid I \subseteq \mR\times \{2\} \}$.
    \item $\BB\coloneq  \{B\}$.
    \item $\LL \coloneq  \{ I \in \mathbb{I}(B) \setminus \BB \mid   -\infty \in I  \}$. 
    \item $\RR\coloneq \{ I \in \mathbb{I}(B) \setminus \BB \mid +\infty \in I  \}$.
\end{itemize}
 Let $\TT := \{\UU,\DD,\BB,\LL,\RR \}$. Every interval of $B$ belongs to $\XX$ for some $\XX\in \TT$. 
 We say that intervals $I$ and $J$ are the \emph{same type} if $I$ and $J$ are in common  $\XX$ for $\XX \in \TT$. Otherwise, we say that $I$ and $J$ are not the same type of intervals. 

  We next describe intervals by using decorated numbers.
 Let $\omR:= \mR \sqcup \{\pm \infty\}$. 
 \emph{Decorated numbers} are elements of the following set.
\begin{equation}
\overline{ \mathbb{R}}^* \coloneqq \{r^+ \mid r \in \overline{ \mathbb{R}} \} \sqcup \{r^- \mid r \in \ \overline{ \mathbb{R}} \}.
\end{equation}\label{eq:decorated-R}  
The set $\overline{ \mathbb{R}}^*$ is equipped with a total order $\leq_*$ defined as follows: for any $s,t \in \omR$,
\[
s^\sigma \leq_* t^\tau \text{ if either $s<t$, or $s=t$ with  $(\sigma,\tau) \in  \{ (- , -) , (-,+), (+,+) \} $.}
\]
Then, any interval of $\omR$ is described by a pair of decorated numbers $s^\sigma$ and $t^\tau$
as follows: \begin{equation}
( s^\sigma,t^\tau ) \coloneqq
    \begin{cases} 
        [s,t] & \text{ if $(\sigma,\tau)=(-,+)$}.\\
         [s,t) & \text{ if $(\sigma,\tau)=(-,-)$}.\\
         (s,t] & \text{ if $(\sigma,\tau)=(+,+)$}.\\
          (s,t) & \text{ if $(\sigma,\tau)=(+,-)$}.
    \end{cases}
\end{equation} 
We also define $\mR^* := \omR^*\setminus \{ - \infty^\pm, +\infty^\pm \}$. Using the above notation, we can describe $\UU,\DD,\LL$, and $\RR$ as follows.
\begin{itemize}
\item $\UU = \{ ( s^\sigma, t^\tau ) \times \{1\} \mid s^\sigma<_* t^\tau\in  \overline{\mathbb{R}}^* \setminus \{-\infty^- , +\infty^+\} \}$.
\item $\DD=\{  ( t^\tau, s^\sigma) \times \{2\} \mid t^\tau<_* s^\sigma\in  \overline{\mathbb{R}}^* \setminus \{-\infty^- , +\infty^+\}  \}$.
\item $\LL =\{ [-\infty,t^\tau ) \times\{1\} \cup [-\infty,s^\sigma ) \times \{2\} \mid  s^\sigma, t^\tau\in  \overline{\mathbb{R}}^* \setminus \{+ \infty^+\}$.
    \item $\RR= \{(s^\sigma,+\infty ] \times\{1\} \cup (t^\tau, +\infty ] \times \{2\}\mid  s^\sigma, t^\tau\in  \overline{\mathbb{R}}^* \setminus \{- \infty^-\} \}$.
\end{itemize}
Each interval $I\in \XX$ with $\XX \in \TT \setminus \{\BB\}$  is given by one of the above forms using $s^\sigma$ and $t^\tau$, and we denote it by $\langle s^\sigma, t^\tau \rangle_\XX$.

For our purpose, we decompose $\LL$ into ${\LL_i}$'s as follows: 
\begin{enumerate}
    \item $\LL_1 \coloneq \{  \langle -\infty^+ ,-\infty^+ \rangle_\LL \} =  \{-\infty\}$.
    \item $\LL_2 \coloneq \{ \langle -\infty^+ ,t^\tau \rangle_\LL  \in \mathbb{I}(B) \mid t^\tau\in \mR^* \cup \{ +\infty^- \} \} \cup \{ \langle s^\sigma, -\infty^+ \rangle_\LL  \in \mathbb{I}(B) \mid  s^\sigma \in \mR^* \cup \{ +\infty^- \} \}$.
    \item $\LL_3 \coloneq \{ \langle s^\sigma,t^\tau \rangle_\LL  \in \mathbb{I}(B)  \mid  s^\sigma,t^\tau \in \mR^*\}$.
    \item $\LL_4 \coloneq \{ \langle + \infty^- ,t^\tau \rangle_\LL  \in \mathbb{I}(B) \mid  t^\tau \in \mR^* \} \cup \{ \langle s^\sigma,+ \infty^-  \rangle_\LL  \in \mathbb{I}(B) \mid  s^\sigma  \in \mR^* \}$.
    \item $\LL_5 \coloneq \{  \langle +\infty^- , +\infty^- \rangle_\LL \} = \{ B \setminus \{ +\infty\}  \}$.
\end{enumerate}
As for $\RR$, we have the following. 
\begin{enumerate}
    \item $\RR_1 \coloneq \{  \langle +\infty^- ,+\infty^- \rangle_\RR \} =  \{+\infty\}$.
    \item $\RR_2 \coloneq \{ \langle +\infty^- ,t^\tau \rangle_\RR  \in \mathbb{I}(B) \mid t^\tau\in \mR^* \cup \{ -\infty^+ \} \} \cup \{ \langle s^\sigma, +\infty^- \rangle_\RR  \in \mathbb{I}(B) \mid  s^\sigma \in \mR^* \cup \{ -\infty^+ \} \}$.
    \item $\RR_3 \coloneq \{ \langle s^\sigma,t^\tau \rangle_\RR  \in \mathbb{I}(B)  \mid  s^\sigma,t^\tau \in \mR^*\}$.
    \item $\RR_4 \coloneq \{ \langle - \infty^+ ,t^\tau \rangle_\RR  \in \mathbb{I}(B) \mid  t^\tau \in \mR^* \} \cup \{ \langle s^\sigma, -\infty^+  \rangle_\RR  \in \mathbb{I}(B) \mid  s^\sigma  \in \mR^* \}$.
    \item $\RR_5 \coloneq \{  \langle -\infty^+ , -\infty^+ \rangle_\RR \} = \{ B \setminus \{ -\infty\}  \}$.
\end{enumerate}
See Table \ref{tab:intervalsBLi}.
\begin{table}[h]
\renewcommand{\arraystretch}{1.5}
    \begin{tabular}{cccccccc}
        $\LL_1$ & \hspace{-2mm}$\LL_2$ & \hspace{-2mm}$\LL_3$ & \hspace{-2mm} $\LL_4$& \hspace{-2mm}$\LL_5$  \\ 
        \begin{tikzpicture}[baseline = 0mm, scale = 0.7]
        \coordinate (x) at (0.4,0);
        \coordinate (y) at (0,0.6); 
        \coordinate (s) at ($3*(x) + -1*(y)$);
        \coordinate (0) at ($0*(x) + 0*(y)$);
        \coordinate (1) at ($1*(x) + 1*(y) - 0.1*(x)$);
        \coordinate (n) at ($5*(x) + 1*(y)$);
        \coordinate (11) at ($6*(x) + 0*(y) - 0.1*(x)$);
        \coordinate (1') at ($1*(x) + -1*(y) - 0.1*(x)$);
        \coordinate (m') at ($5*(x) + -1*(y)$);
        \draw (0)--(1)--(n)--(11)--(m')--(1')--cycle;
        \node[left] at (0) {\footnotesize$-\infty$};
        \node[right] at (11) {\footnotesize$+\infty$}; 
        \node[above] at (1) {\footnotesize$ $}; 
        \node[above] at (n) {\footnotesize$ $}; 
        \node[below] at (1') {\footnotesize$ $}; 
        \node[below] at (m') {\footnotesize$ $};
            \fill (0)circle(0.7mm)node[below]{$ $};
            \fill
            (s)circle(0mm)node[below]{$\langle -\infty^+, -\infty^+ \rangle_\LL$};
        \end{tikzpicture}

        & \hspace{-4mm}
        \begin{tikzpicture}[baseline = 0mm, scale =0.7]
        \coordinate (x) at (0.4,0);
        \coordinate (y) at (0,0.6); 
         \coordinate (s) at ($3*(x) + -1*(y)$);
        \coordinate (0) at ($0*(x) + 0*(y)$);
        \coordinate (1) at ($1*(x) + 1*(y) - 0.1*(x)$);
        \coordinate (n) at ($5*(x) + 1*(y)$);
        \coordinate (11) at ($6*(x) + 0*(y) - 0.1*(x)$);
        \coordinate (1') at ($1*(x) + -1*(y) - 0.1*(x)$);
        \coordinate (m') at ($5*(x) + -1*(y)$);
        \draw (0)--(1)--(n)--(11)--(m')--(1')--cycle;
        \coordinate (s) at ($3*(x) + -1*(y)$);
        \coordinate (t) at ($3*(x) + 1*(y)$);
        \fill (0)circle(0.5mm)node[below]{$ $};
        \fill (t)circle(0.5mm)node[above]{$t^\tau$};
        \draw[line width = 0.7mm] (0)--(1)--(t);
        \node[left] at (0) {\footnotesize$-\infty$};
        \node[right] at (11) {\footnotesize$+\infty$}; 
        \node[above] at (1) {\footnotesize$ $}; 
        \node[above] at (n) {\footnotesize$ $}; 
        \node[below] at (1') {\footnotesize$ $};
        \node[below] at (m') {\footnotesize$ $};
            \fill
            (s)circle(0mm)node[below]{$\langle -\infty^+, t^\tau \rangle_\LL$};
        \end{tikzpicture} 
        & \hspace{-4mm}

           \begin{tikzpicture}[baseline = 0mm, scale =0.7]
        \coordinate (x) at (0.4,0);
        \coordinate (y) at (0,0.6); 
        \coordinate (0) at ($0*(x) + 0*(y)$);
        \coordinate (1) at ($1*(x) + 1*(y) - 0.1*(x)$);
        \coordinate (n) at ($5*(x) + 1*(y)$);
        \coordinate (11) at ($6*(x) + 0*(y) - 0.1*(x)$);
        \coordinate (1') at ($1*(x) + -1*(y) - 0.1*(x)$);
        \coordinate (m') at ($5*(x) + -1*(y)$);
        \draw (0)--(1)--(n)--(11)--(m')--(1')--cycle;
        \coordinate (s) at ($3*(x) + -1*(y)$);
        \coordinate (t) at ($3*(x) + 1*(y)$);
        \fill (s)circle(0.5mm)node[above]{$s^\sigma$};
        \fill
        (s)circle(0mm)node[below]{$\langle s^\sigma, t^\tau \rangle_\LL$};
        \fill (t)circle(0.5mm)node[above]{$t^\tau$};
        \draw[line width = 0.7mm] (s)--(1')--(0)--(1)--(t);
        \node[left] at (0) {\footnotesize$-\infty$};
        \node[right] at (11) {\footnotesize$+\infty$}; 
        \node[above] at (1) {\footnotesize$ $}; 
        \node[above] at (n) {\footnotesize$ $}; 
        \node[below] at (1') {\footnotesize$ $};
        \node[below] at (m') {\footnotesize$ $};
        \end{tikzpicture}

        & \hspace{-4mm}
       \begin{tikzpicture}[baseline = 0mm, scale =0.7]
        \coordinate (x) at (0.4,0);
        \coordinate (y) at (0,0.6); 
        \coordinate (0) at ($0*(x) + 0*(y)$);
        \coordinate (1) at ($1*(x) + 1*(y) - 0.1*(x)$);
        \coordinate (n) at ($5*(x) + 1*(y)$);
        \coordinate (11) at ($6*(x) + 0*(y) - 0.1*(x)$);
        \coordinate (11-up) at ($6*(x) + 0.15*(y) - 0.25*(x)$);
        \coordinate (1') at ($1*(x) + -1*(y) - 0.1*(x)$);
        \coordinate (m') at ($5*(x) + -1*(y)$);
        \draw (0)--(1)--(n)--(11)--(m')--(1')--cycle;
        \coordinate (s) at ($3*(x) + -1*(y)$);
        \coordinate (t) at ($3*(x) + 1*(y)$);
        \fill (s)circle(0.5mm)node[above]{$s^\sigma$};
        \fill (11-up)circle(0.5mm)node[above right]{$ $};
        \draw[line width = 0.7mm] (s)--(1')--(0)--(1)--(t)--(n)--(11-up);
        \node[left] at (0) {\footnotesize$-\infty$};
        \node[right] at (11) {\footnotesize$+\infty$}; 
        \node[above] at (1) {\footnotesize$ $}; 
        \node[above] at (n) {\footnotesize$ $}; 
        \node[below] at (1') {\footnotesize$ $};
        \node[below] at (m') {\footnotesize$ $};
                \fill
            (s)circle(0mm)node[below]{$\langle s^\sigma, +\infty^- \rangle_\LL$};
        \end{tikzpicture}  
        & 
        
        \hspace{-4mm}
            \begin{tikzpicture}[baseline = 0mm, scale =0.7]
        \coordinate (x) at (0.4,0);
        \coordinate (y) at (0,0.6); 
        \coordinate (0) at ($0*(x) + 0*(y)$);
        \coordinate (1) at ($1*(x) + 1*(y) - 0.1*(x)$);
        \coordinate (n) at ($5*(x) + 1*(y)$);
        \coordinate (11) at ($6*(x) + 0*(y) - 0.1*(x)$);
        \coordinate (11-up) at ($6*(x) + 0.15*(y) - 0.25*(x)$);
        \coordinate (11-down) at ($6*(x) - 0.15*(y) - 0.25*(x)$);
        \coordinate (1') at ($1*(x) + -1*(y) - 0.1*(x)$);
        \coordinate (m') at ($5*(x) + -1*(y)$);
        \draw (0)--(1)--(n)--(11)--(m')--(1')--cycle;
        \coordinate (s) at ($3*(x) + -1*(y)$);
        \coordinate (t) at ($3*(x) + 1*(y)$);
        \fill (11-up)circle(0.5mm)node[above right]{$ $};
        \draw[line width = 0.7mm] (11-down)--(m')--(1')--(0)--(1)--(t)--(n)--(11-up);
        \node[left] at (0) {\footnotesize$-\infty$};
        \node[right] at (11) {\footnotesize$+\infty$}; 
        \node[above] at (1) {\footnotesize$ $}; 
        \node[above] at (n) {\footnotesize$ $}; 
        \node[below] at (1') {\footnotesize$ $};
        \node[below] at (m') {\footnotesize$ $};
              \fill
            (s)circle(0mm)node[below]{$\langle +\infty^-, +\infty^- \rangle_\LL$};
        \end{tikzpicture}       
        \\      
                $\RR_1$ & \hspace{-2mm}$\RR_2$ & \hspace{-2mm}$\RR_3$ & \hspace{-2mm} $\RR_4$& \hspace{-2mm}$\RR_5$
                \\
         \begin{tikzpicture}[baseline = 0mm, scale = 0.7]
        \coordinate (x) at (0.4,0);
        \coordinate (y) at (0,0.6); 
        \coordinate (0) at ($0*(x) + 0*(y)$);
        \coordinate (1) at ($1*(x) + 1*(y) - 0.1*(x)$);
        \coordinate (n) at ($5*(x) + 1*(y)$);
        \coordinate (11) at ($6*(x) + 0*(y) - 0.1*(x)$);
        \coordinate (1') at ($1*(x) + -1*(y) - 0.1*(x)$);
        \coordinate (m') at ($5*(x) + -1*(y)$);
        \draw (0)--(1)--(n)--(11)--(m')--(1')--cycle;
        \node[left] at (0) {\footnotesize$-\infty$};
        \node[right] at (11) {\footnotesize$+\infty$}; 
        \node[above] at (1) {\footnotesize$ $}; 
        \node[above] at (n) {\footnotesize$ $}; 
        \node[below] at (1') {\footnotesize$ $}; 
        \node[below] at (m') {\footnotesize$ $};
            \fill (11)circle(0.7mm)node[below]{$ $};
                        \fill
            (s)circle(0mm)node[below]{$\langle +\infty^-, +\infty^- \rangle_\RR$};
        \end{tikzpicture}           
          & 
          \hspace{-4mm} 
               \begin{tikzpicture}[baseline = 0mm, scale =0.7]
        \coordinate (x) at (0.4,0);
        \coordinate (y) at (0,0.6); 
        \coordinate (0) at ($0*(x) + 0*(y)$);
        \coordinate (1) at ($1*(x) + 1*(y) - 0.1*(x)$);
        \coordinate (n) at ($5*(x) + 1*(y)$);
        \coordinate (11) at ($6*(x) + 0*(y) - 0.1*(x)$);
        \coordinate (0_down) ($0.2*(x) - 0.5*(y)$);
        \coordinate (1') at ($1*(x) + -1*(y) - 0.1*(x)$);
        \coordinate (m') at ($5*(x) + -1*(y)$);
        \draw (0)--(1)--(n)--(11)--(m')--(1')--cycle;
        \coordinate (s) at ($3*(x) + -1*(y)$);
        \coordinate (t) at ($3*(x) + 1*(y)$);
         \fill (11)circle(0.5mm)node[below]{$ $};
        \draw[line width = 0.7mm] (s)--(m')--(11);
        \node[left] at (0) {\footnotesize$-\infty$};
        \node[right] at (11) {\footnotesize$+\infty$}; 
        \node[above] at (1) {\footnotesize$ $}; 
        \node[above] at (n) {\footnotesize$ $}; 
        \node[below] at (1') {\footnotesize$ $};
        \node[below] at (m') {\footnotesize$ $};
         \fill (s)circle(0.5mm)node[above]{$t^\tau$};
             \fill
            (s)circle(0mm)node[below]{$\langle +\infty^-, t^\tau \rangle_\RR$};
        \end{tikzpicture}  
        & 
        \hspace{-4mm} 
          \begin{tikzpicture}[baseline = 0mm, scale =0.7]
        \coordinate (x) at (0.4,0);
        \coordinate (y) at (0,0.6); 
        \coordinate (0) at ($0*(x) + 0*(y)$);
        \coordinate (1) at ($1*(x) + 1*(y) - 0.1*(x)$);
        \coordinate (n) at ($5*(x) + 1*(y)$);
        \coordinate (11) at ($6*(x) + 0*(y) - 0.1*(x)$);
        \coordinate (0_down) ($0.2*(x) - 0.5*(y)$);
        \coordinate (1') at ($1*(x) + -1*(y) - 0.1*(x)$);
        \coordinate (m') at ($5*(x) + -1*(y)$);
        \draw (0)--(1)--(n)--(11)--(m')--(1')--cycle;
        \coordinate (s) at ($3*(x) + -1*(y)$);
        \coordinate (t) at ($3*(x) + 1*(y)$);
        \draw[line width = 0.7mm] (t)--(n)--(11)--(m')--(s);
        \node[left] at (0) {\footnotesize$-\infty$};
        \node[right] at (11) {\footnotesize$+\infty$}; 
        \node[above] at (1) {\footnotesize$ $}; 
        \node[above] at (n) {\footnotesize$ $}; 
        \node[below] at (1') {\footnotesize$ $};
        \node[below] at (m') {\footnotesize$ $};
         \fill (s)circle(0.5mm)node[above]{$t^\tau$};
        \fill (t)circle(0.5mm)node[above]{$s^\sigma$};
          \fill
        (s)circle(0mm)node[below]{$\langle s^\sigma, t^\tau \rangle_\RR$};
        \end{tikzpicture} 
        &
        \hspace{-4mm} 
                \begin{tikzpicture}[baseline = 0mm, scale =0.7]
        \coordinate (x) at (0.4,0);
        \coordinate (y) at (0,0.6); 
        \coordinate (0) at ($0*(x) + 0*(y)$);
        \coordinate (0_down) at ($0.1*(x) - 0.1*(y)$);
        \coordinate (1) at ($1*(x) + 1*(y) - 0.1*(x)$);
        \coordinate (n) at ($5*(x) + 1*(y)$);
        \coordinate (11) at ($6*(x) + 0*(y) - 0.1*(x)$);
        \coordinate (11-up) at ($6*(x) + 0.15*(y) - 0.25*(x)$);
        \coordinate (11-down) at ($6*(x) - 0.15*(y) - 0.25*(x)$);
        \coordinate (1') at ($1*(x) + -1*(y) - 0.1*(x)$);
        \coordinate (m') at ($5*(x) + -1*(y)$);
        \draw (0)--(1)--(n)--(11)--(m')--(1')--cycle;
        \coordinate (s) at ($3*(x) + -1*(y)$);
        \coordinate (t) at ($3*(x) + 1*(y)$);
        \fill (t)circle(0.5mm)node[above]{$ $};
        \draw[line width = 0.7mm] (0_down)--(1')--(m')--(11)--(n)--(t);
        \node[left] at (0) {\footnotesize$-\infty$};
        \node[right] at (11) {\footnotesize$+\infty$}; 
        \node[above] at (t) {\footnotesize$ s^\sigma$}; 
        \node[above] at (n) {\footnotesize$ $}; 
        \node[below] at (1') {\footnotesize$ $};
        \node[below] at (m') {\footnotesize$ $};
          \fill
        (s)circle(0mm)node[below]{$\langle s^\sigma, -\infty^+ \rangle_\RR$};
        \end{tikzpicture} 
        &
        \hspace{-4mm}  
             \begin{tikzpicture}[baseline = 0mm, scale =0.7]
        \coordinate (x) at (0.4,0);
        \coordinate (y) at (0,0.6); 
        \coordinate (0) at ($0*(x) + 0*(y)$);
        \coordinate (0_down) at ($0.1*(x) - 0.1*(y)$);
        \coordinate (0_up) at ($0.1*(x) + 0.1*(y)$);
        \coordinate (1) at ($1*(x) + 1*(y) - 0.1*(x)$);
        \coordinate (n) at ($5*(x) + 1*(y)$);
        \coordinate (11) at ($6*(x) + 0*(y) - 0.1*(x)$);
        \coordinate (11-up) at ($6*(x) + 0.15*(y) - 0.25*(x)$);
        \coordinate (11-down) at ($6*(x) - 0.15*(y) - 0.25*(x)$);
        \coordinate (1') at ($1*(x) + -1*(y) - 0.1*(x)$);
        \coordinate (m') at ($5*(x) + -1*(y)$);
        \draw (0)--(1)--(n)--(11)--(m')--(1')--cycle;
        \coordinate (s) at ($3*(x) + -1*(y)$);
        \coordinate (t) at ($3*(x) + 1*(y)$);
        \fill (11-up)circle(0.5mm)node[above right]{$ $};
          \fill
        (s)circle(0mm)node[below]{$\langle -\infty^+, -\infty^+ \rangle_\RR$};
        \draw[line width = 0.7mm] (0_down)--(1')--(m')--(11)--(n)--(1)--(0_up);
        \node[left] at (0) {\footnotesize$-\infty$};
        \node[right] at (11) {\footnotesize$+\infty$}; 
        \node[above] at (1) {\footnotesize$ $}; 
        \node[above] at (n) {\footnotesize$ $}; 
        \node[below] at (1') {\footnotesize$ $};
        \node[below] at (m') {\footnotesize$ $};
        \end{tikzpicture}   
    \end{tabular}
    \caption{ Intervals in $\LL_i$ and  $\RR_i$ ($i=1,2,3,4,5$). We represent each interval with thick lines. We omit writing a picture of intervals $\langle s^\sigma , -\infty^+ \rangle_\LL \in \LL_2$ and  $\langle +\infty^- , t^\tau \rangle_\LL \in \LL_4$; and $\langle s^\sigma , +\infty^- \rangle_\RR \in \RR_2$ and  $\langle -\infty^+ , t^\tau \rangle_\RR \in \RR_4$.
    }
    \label{tab:intervalsBLi}
\end{table}

One can consider the bipath persistent homology of a $B$-filtration as analogue of the case of finite bipath posets. Next, we introduce bipath functions, which can induce bipath persistent homology.

\begin{definition}\label{def:bipathfunction}
Let $X$ be a topological space.
A \emph{bipath function} on $X$ is a pair of $B$-valued maps $(f_1,f_2)$ on $X$ such that  $f_i\colon X \to (\mR \times \{i\}) \sqcup \{\pm \infty \} \subseteq B$ and  $f_1^{-1}(\{-\infty\}) = f_2^{-1}(\{-\infty\})$. 
\end{definition}

We set $| \cdot, \cdot|_B \colon B\times B \to \mR_{\geq0} \sqcup\{\infty\}$  by
\begin{equation*}
    |b_1 , b_2|_B \coloneq
    \begin{cases}
        0 & \text{ if } b_1 =b_2,\\ 
        |r_1-r_2| & \text{ if } b_1 = (r_1,i), b_2 =(r_2,i) \in \mR \times \{i\} \text{ for $i=1,2$, }\\
        \infty & \text{ else.}
    \end{cases}
\end{equation*}
Using this notation, for two bipath functions $f=(f_1,f_2)$ and $g=(g_1,g_2)$ on a topological space $X$, we define 
\begin{equation*}
  ||f,g||_{B} := \max \{  \sup_{x\in X}  |f_1(x) ,g_1(x)|_B , \sup_{x\in X}  |f_2(x) ,g_2(x)|_B\} .   
\end{equation*}

Next, we define a sublevelset bipath filtration of a bipath function.

\begin{definition}\label{def:bipath_function}
Let $X$ be a topological space. For a bipath function on $X$
$f=(f_1,f_2)$, 
we have a $B$-filtration $(f\leq \cdot)\colon B \to \Top$ defined by
     \begin{equation*}
       (f\leq b) \coloneqq
        \begin{cases}
         f_1^{-1}(\{- \infty\}) & \text{if } b= - \infty,
         \\
         X & \text{if } b= + \infty,
         \\
        \{x \in X \ | \  f_i(x) \leq r \} & \text{if $b =(r,i) \in \mathbb{R}\times \{i\}$} \text{ for $i=1,2$}, 
        \end{cases}
\end{equation*}
for each $b \in B$. In fact $(f\leq b_1) \subseteq (f\leq b_2)$ holds whenever $b_1 \leq b_2$. We call it \emph{sublevelset bipath filtration} of $f$. 
\end{definition}

For a bipath function $f=(f_1,f_2)$ on a topological space $X$, the sublevelset bipath filtration $(f\leq \cdot)$ can be regarded as a pair of sublevelset filtration $f_i \colon \omR  \to X $ ($i=1,2$) such that $(f_1\leq - \infty ) = (f_2\leq - \infty )$.

Bipath functions induce bipath persistent homology. Let $X$ be a topological space and $f$ be a bipath function on $X$. We call
\[V(f) \coloneq  H_q( (f\leq \cdot) ; k) \in \Rep_k(B)\]
the \emph{bipath persistent homology} of $f$. We say that a bipath function $f$ is \emph{tame} if $V(f) \in \pfdk(B)$.

%% file: 4_isometry.tex
\section{Stability theorem of bipath persistent homology}\label{section:stability}
In this section, we prove the stability theorem of bipath persistent homology of bipath functions (Theorem \ref{thm:stability}). We will deduce it by using the isometry theorem of pfd bipath persistence modules (Theorem \ref{thm:isometry}). 

Let $B$ be the bipath poset. 
We define an $\mR_{\geq0}$-action $\Lambda^B$ on $B$ by translations $\Lambda_\epsilon^B \colon B \to B$ given by
\begin{equation}\label{eq:superfamilyB}
        \Lambda_\epsilon^B(b) \coloneqq
        \begin{cases}
        (r+\epsilon,i) & \text{if $b =(r,i) \in \mathbb{R}\times \{i\}$} \text{ for $i=1,2$,}\\
        \pm \infty & \text{if } b= \pm \infty,
        \end{cases}
    \end{equation} 
    for each $b \in B$. Since $\Lambda_\epsilon^B $ is an automorphism for any $\epsilon\geq 0$ by definition, it gives an $\mR$-action on $B$. 

This section aims to prove the following stability result on the bipath persistent homology with respect to bipath functions and the above $\mR$-action $\Lambda^B$ on $B$. This theorem suggests the stability of bipath persistence diagrams of bipath functions.

\begin{theorem}[stability theorem of bipath persistent homology]\label{thm:stability} 
    Let $f=(f_1,f_2)$ and $g=(g_1,g_2)$ be tame bipath functions on a topological space $X$ such that
    \begin{equation}\label{eq:bottom_condition}
    f_1^{-1}(\{-\infty\}) = f_2^{-1}(\{-\infty\}) = g_1^{-1}(\{-\infty\}) =g_2^{-1}(\{-\infty\}).    
    \end{equation}
     Then we have the following inequality.
    \begin{equation*}
        d_{\rm{B}}^{\Lambda^B}(\mathscr{B}(V(f)),\mathscr{B}(V(g)) ) \leq \|f,g \|_{B}.
    \end{equation*}
\end{theorem}

To prove Theorem \ref{thm:stability}, we use 
Theorem \ref{thm:isometry}.

\begin{theorem}[isometry theorem of bipath persistence modules]\label{thm:isometry}
    Let $V$ and $W$ be pfd bipath persistence modules. For  $\epsilon\geq 0$, the following conditions are equivalent.
    \begin{enumerate}
        \item[ {\rm (a)} ] There exists a $\Lambda_\epsilon^{B}$-matching between $\mathscr{B}(V)$ and $\mathscr{B}(W)$.
        \item[ {\rm (b)} ] There exists a bottleneck $\Lambda^B_\epsilon$-interleaving between $V$ and $W$.
        \item[  {\rm (c)} ] There exists a $\Lambda^B_\epsilon$-interleaving between $V$ and $W$.
    \end{enumerate}  Therefore, we have the following equality.
    \begin{equation*}
        d_{\rm{I}}^{\Lambda^B}(V,W) = d_{\rm{BI}}^{\Lambda^B}(V,W) = d_{\rm{B}}^{\Lambda^B}(\mathscr{B}(V),\mathscr{B}(W)).
    \end{equation*}
\end{theorem}

We call the implication (c) $\Rightarrow$ (b) the algebraic stability theorem of bipath persistence modules.

Once we admit Theorem \ref{thm:isometry}, we can prove Theorem \ref{thm:stability} in the following way. The proof is similar to a proof of the stability theorem of standard persistent homology. 

\begin{proof}[Proof of Theorem ~\ref{thm:stability}]
   Let $f$ and $g$ be tame bipath functions. If $\|f,g \|_{B} = + \infty$, then we have $d_{\rm{B}}^{\Lambda^B}(V(f),V(g)) \leq +\infty = \|f,g \|_{B}$. 
   
   Next, we assume $ \epsilon \coloneqq \|f,g \|_{B} < +\infty$. Then, for any $b \in B$, we have 
\begin{equation}\label{eq:inclu}
       (f\leq b) \subseteq (g \leq \Lambda^B_{\epsilon}(b) ) \text{ and } (g\leq b) \subseteq (f \leq \Lambda^B_{\epsilon}(b) ).
   \end{equation}
  Indeed, if $b\neq - \infty$, then we can easily check \eqref{eq:inclu}.
    If $b = -\infty$, then by the assumption of bipath functions \eqref{eq:bottom_condition}, we have $(f \leq -\infty) = (g \leq -\infty)$. In addition, since $-\infty$ is a fixed point of $\Lambda^B_\epsilon$, we obtain \eqref{eq:inclu}.

Thus, we have the following diagram of topological spaces:
\begin{equation*}
\xymatrix{
(f\leq b) \ar@{->}[r] \ar@{->}[rd] 
& (f \leq \Lambda^B_{\epsilon}(b) ) \ar@{->}[r] \ar@{->}[rd] 
& (f \leq \Lambda^B_{2\epsilon}(b) )   \\
(g\leq b) \ar@{->}[r] \ar@{->}[ru] 
& (g \leq \Lambda^B_{\epsilon}(b) ) \ar@{->}[r] \ar@{->}[ru] 
& (g \leq \Lambda^B_{2\epsilon}(b) ).  
}    
\end{equation*}
By applying the $q$th homology functor to the above diagram, we obtain the following diagram of vector spaces for each $b\in B$:
\begin{equation*}
\xymatrix{
V(f)_b \ar@{->}[r] \ar@{->}[rd] 
& V(f)_{ \Lambda^B_{\epsilon}(b) } \ar@{->}[r] \ar@{->}[rd] 
& V(f)_{\Lambda^B_{2\epsilon}(b) }   \\
V(g)_b \ar@{->}[r] \ar@{->}[ru] 
& V(g)_{ \Lambda^B_{\epsilon}(b) } \ar@{->}[r] \ar@{->}[ru] 
& V(g)_{ \Lambda^B_{2\epsilon}(b) }.  
}    
\end{equation*}
Then, they give rise to morphisms
\begin{eqnarray*}
 \alpha \colon  V(f)\to V(g)(\epsilon) \text{ and } 
 \beta  \colon  V(g) \to V(f)(\epsilon)   
\end{eqnarray*}
such that  $\beta(\epsilon)\circ \alpha = V(f)_{0\to2\epsilon}$ and $\alpha(\epsilon)\circ \beta = V(g)_{0\to2\epsilon}$. Thus, the pair $\alpha$ and $\beta$ is a $\Lambda^B_\epsilon$-interleaving between $V(f)$ and $V(g)$. Therefore, we have $d_{\rm{I}}^{\Lambda^B}(V(f), V(g)) \leq  \epsilon$.  By combining this with Theorem~\ref{thm:isometry}, we obtain the desired inequality
 \[d_{\rm{B}}^{\Lambda^B} (\mathscr{B}(V(f)), \mathscr{B}(V(g))) = d_{\rm{I}}^{\Lambda^B}(V(f), V(g)) \leq  \epsilon = \|f,g\|_{B}.\]
\end{proof}

In the rest of this paper, we prove Theorem \ref{thm:isometry}. 
We have the implication (a) $\Rightarrow$ (b) and (b) $\Rightarrow$ (c) in the statement of Theorem \ref{thm:isometry} by Propositions \ref{lem:inequality} and \ref{prop:dI<dB}. Below, we show their converses. For simplicity, we write $\Lambda$ instead of the $\mR$-action $\Lambda^B$ on $B$ throughout the following subsections.

\subsection{Interleaving and partial matching for bipath persistence modules}\label{subsec:mor_interval} 
 
In this subsection, we give basic properties of interleavings and partial matching and prove the implication (b) $\Rightarrow$ (a) in the statements of Theorem \ref{thm:isometry}. 

Recall that for the division $\mathbb{I}(B) =  \UU \sqcup \DD  \sqcup \BB \sqcup \LL \sqcup  \RR$, we write $\TT$ for the set $\{\UU,\DD,\BB,\LL,\RR \}$ (Section \ref{sec:bipath}). 
 For any pfd bipath persistence module $V$, we decompose $V \cong \bigoplus_{\XX \in \TT} V^{\XX}$, where $V^{\XX}$ is a direct sum of intervals in $\mathscr{B}(V)$ belonging to $\XX$ (Such decomposition is guaranteed by Theorem~\ref{thm:decomp}). Notice that $V^\XX$ is uniquely determined up isomorphism for each $\XX \in \TT$.

\begin{lemma}\label{lem:type_shiftfunctor}
   For any interval $I$ of $B$  and $\epsilon \geq 0$,  the intervals $I$ and  $\Lambda_{-\epsilon}(I)$  are the same type of intervals. 
    Therefore, for any pfd bipath persistence module $V= \bigoplus_{\XX \in \TT} V^{\XX}$, we have $V^\XX(\epsilon) = V(\epsilon)^\XX$ for each $\XX \in \TT$.
\end{lemma}

  \begin{proof}
      It is obvious by the definition of a translation $\Lambda_\epsilon$.
  \end{proof}

 We next define a partial order $\leq_{\TT}$ on $\TT$ so that the Hasse diagram is given by  \begin{equation*}
\begin{tikzpicture}
    \coordinate (RR) at (-2,0);
    \coordinate (TT) at (-2.5,0);
     \coordinate (RR1) at (-2,0.2);
      \coordinate (RR2) at (-2,-0.2);
    \coordinate (UU) at (0,1.0);  
     \coordinate (UU1) at (0,0.9); 
    \coordinate (UUm) at (-0.6,0.9); 
     \coordinate (BB) at (0,0); 
     \coordinate (BBm) at (-0.5,0.0); 
    \coordinate (DD) at (0,-1.0);
     \coordinate (DD1) at (0,-0.9);
    \coordinate (DDm) at (-0.6,-0.9);
    \coordinate (LL) at (2,0);
    \coordinate (LLm1) at (1.3,0.3);
    \coordinate (LLm2) at (1.3,-0.3);
    \coordinate (LLm) at (1.3,0.0);
    \node at ($(TT)$) [left]{$\TT\colon$};
   \node at ($(RR)$) [left]{$\RR$}; 
      \node at ($(UU)$) [left]{$\UU$}; 
   \node at ($(BB)$) [left]{$\BB$}; 
   \node at ($(DD)$) [left]{$\DD$}; 
   \node at ($(LL)$) [left]{$\LL$.}; 

\draw [->] (RR) -- (BBm) ;
\draw [->] (RR1) -- (UUm) ;
\draw [->] (RR2) -- (DDm) ;
\draw [->] (BB) -- (LLm) ;
\draw [->] (UU1) --(LLm1) ;
\draw [->] (DD1) -- (LLm2) ;
\end{tikzpicture}    
\end{equation*}

As for the partial order $\leq_\TT$ on $\TT$, we have the following observation.

\begin{lemma}\label{lem:map-types}
Let $\XX,\YY \in \TT$ with $\XX \nleq_\TT \YY$. Then, for any $I \in \XX$ and $J \in \YY$, we have $\Hom_{B} (k_I,k_J) =0$. 
\end{lemma}

\begin{proof}
Let $I$ and $J$ be intervals in $\XX$ and $\YY$, respectively such that $\XX \nleq_\TT \YY$. Then, 
    we can see $\Omega(I,J)=\emptyset$. By Lemma \ref{lem:well-def}, we have $\Hom_{B}(k_I,k_J)=0$.
\end{proof}

\begin{corollary}\label{rem:become-zero}
Let $I_1$, $I_2$, and $J$ be intervals of $B$. If $I_1$ and $I_2$ are of the same type, and if $J$ and $I_i$ are not of the same for $i=1,2$, then any composition of morphisms $k_{I_1} \to k_J \to k_{I_2}$ is the zero morphism.   
\end{corollary}

\begin{proof}
     It follows from Lemma \ref{lem:map-types}.
\end{proof}

Using the above observation, we give the following result, which asserts that an $\Lambda_\epsilon$-interleaving between pfd bipath persistence modules $V$ and $W$ induces an $\Lambda_\epsilon$-interleaving between $V^\XX$ and $W^\XX$ for each $ \XX\in \TT$.
\begin{proposition}\label{lem:type-separating}
  Let $V = \bigoplus_{\XX \in \TT} V^{\XX}$ and $W = \bigoplus_{\XX \in\TT} W^{\XX}$ be pfd bipath persistence modules. For any  $\epsilon\geq 0$, the following statements are equivalent.
  \begin{itemize}
      \item [\rm {(i)} ]$V$ and $W$ are $\Lambda_\epsilon$-interleaved.
      \item [\rm{(ii)} ] $V^{\XX}$ and $W^{\XX}$ are $\Lambda_\epsilon$-interleaved for each $\XX \in\TT$. 
  \end{itemize}   
\end{proposition}

\begin{proof} 
The implication  (ii) $\Rightarrow$ (i) immediately follows from the definition. 

We prove (i) $\Rightarrow$ (ii). 
Let the pair $\alpha \colon V\to W(\epsilon)$ and $\beta \colon W \to V(\epsilon)$ be a $\Lambda_\epsilon$-interleaving between pfd bipath persistence modules $V$ and $W$. Then, they can be written of the form $\alpha =(\alpha_{\YY,\XX}\colon V^{\XX} \to W^{\YY}(\epsilon))_{\XX, \YY \in\TT}$ and $\beta=(\beta_{\YY,\XX}\colon W^{\XX} \to V^{\YY}(\epsilon))_{\XX, \YY \in\TT}$ by Lemma \ref{lem:type_shiftfunctor}. 
Then, by Corollary~\ref{rem:become-zero}, we have
    \begin{eqnarray*}
  (V^{\XX})_{0\to2\epsilon}  &=& \sum_{\YY \in\TT} 
\beta_{\XX, \YY}(\epsilon) \circ  \alpha_{\YY,\XX} = \beta_{\XX, \XX}(\epsilon) \circ  \alpha_{\XX,\XX} \text{ and } \\
  (W^{\XX})_{0\to2\epsilon}  &=& \sum_{\YY \in\TT} 
\alpha_{\XX, \YY}(\epsilon) \circ  \beta_{\YY,\XX} = \alpha_{\XX, \XX}(\epsilon) \circ  \beta_{\XX,\XX} 
\end{eqnarray*}
for each $\XX \in\TT$. This shows that ${V_\XX}$ and ${W_\XX}$ are $\Lambda_\epsilon$-interleaved for each $\XX \in\TT$.
\end{proof}

We additionally study morphisms between interval modules over $B$.

\begin{lemma}\label{lem:mapsLR}
    Let $I$ and $J$ be intervals in $\LL$ (resp. $I,J \in \RR$). Then $\Hom_B(k_I,k_J)$ $\neq$ $0$ if and only if $J \subseteq  I$ $($resp. $I \subseteq  J)$.
\end{lemma}

\begin{proof}
For intervals $I,J$ in $\LL$ (resp. $\RR$),
    by Lemma \ref{lem:well-def}, we have 
    \[
    \Hom_B(k_I,k_J) \neq 0 \iff |\Omega(I,J)| \neq 0 \iff J\subseteq I \text{ (resp.  $ I \subseteq J$), }
    \]
    where the last equivalence follows from a fact that any interval in $\LL$ (resp. $\RR$) contains a minimal element $-\infty\in B$ (resp. a maximum element $+\infty \in B$).
\end{proof}

\begin{lemma}\label{lem:aaa}
    Let $I$ and $J$ be the same type of intervals of $B$. For any $\epsilon\geq0$, we have the following.  
    \begin{enumerate}
        \item [{\rm(1)}]If $I, J \in \XX$ for $\XX \in \{\UU, \DD\}$, then the following statements are equivalent.
        \begin{enumerate}
            \item  [{\rm (i)}] $k_I$ and $k_J$ are $\Lambda_\epsilon$-interleaved.
            \item  [{\rm (ii)}] $k_I$ and $k_J$ are $\Lambda_{2\epsilon}$-trivial, or $J \subseteq \Ex_{\epsilon}^{\Lambda}(I)$ and $I \subseteq \Ex_{\epsilon}^{\Lambda}(J)$.
        \end{enumerate}
        
        \item [{\rm(2)}] For any interval in $\LL\sqcup \RR \sqcup \BB$, the corresponding interval module is $\Lambda_{2\epsilon}$-significant. In addition,
        if $I, J \in \XX$ for $\XX \in \{\LL, \RR, \BB\}$, then the following statements are equivalent.
        \begin{enumerate}
            \item [{\rm (i)}]  $k_I$ and $k_J$ are $\Lambda_\epsilon$-interleaved.
            \item [{\rm (ii)}]  $J \subseteq \Ex_{\epsilon}^{\Lambda}(I)$ and $I \subseteq \Ex_{\epsilon}^{\Lambda}(J)$.
            \item [{\rm (iii)}]   $\Lambda_{-\epsilon}(I) \subseteq J \subseteq \Lambda_\epsilon(I)$.
             \item [{\rm (iv)}]   $\Lambda_{-\epsilon}(J) \subseteq I \subseteq \Lambda_\epsilon(J)$.
        \end{enumerate}
  
    \end{enumerate}
\end{lemma}

\begin{proof}

(1) By identifying $\mR\times \{i\}\  (i=1,2)$ with $\mR$, we regard $\UU$ and $\DD$ as the set of intervals of $\mR$. Then the statement of (1) is a well-known fact, for example, see \cite[p.9]{Les2018AlgebraicStability}.

(2) For any interval of $\LL\sqcup \RR \sqcup \BB$, it contains $-\infty \in B$ or $+\infty \in B$, which are fixed points of $\Lambda_{2\epsilon}$, by definition. 
This implies that the corresponding interval module is $\Lambda_{2\epsilon}$-significant.

If $I = J =B \in \BB$, then the equivalences among (i), (ii), (iii), and (iv) are obvious.

We next consider the case of $\LL$. Notice that any interval $T\in \LL$ contains the minimal element $-\infty\in B$. Thus, we have $\Lambda_\epsilon(T)^{\downarrow} = \Lambda_\epsilon(T)$ and $\Lambda_{-\epsilon}(T)^{\uparrow} = B$. Therefore, we have $\Ex_\epsilon^{\Lambda}(T) = \Lambda_\epsilon(T)$ for any $T\in \LL$.  From this observation, we can deduce the equivalences among (ii), (iii), and (iv). In addition, (ii) $\Rightarrow$ (i) follows from the fact that $\Lambda$ is an $\mR$-action on $B$ and  Proposition~\ref{prop:matching-induce-matching}. We show (i) $\Rightarrow$ (iii). Let a pair $\alpha \colon k_I \to k_J(\epsilon)$ and $\beta \colon  k_J \to k_I(\epsilon)$ be a $\Lambda_\epsilon$-interleaving between $k_I$ and $k_J$.  Since the interval modules $k_I$ and $k_J$ are $\Lambda_{2\epsilon}$-significant, we have $0 \neq (k_I)_{0\to 2\epsilon} =  \beta(\epsilon) \circ \alpha $, which implies $\alpha \neq 0$ and $\beta \neq 0$. 
By Lemmas~\ref{lem:support_interval} and \ref{lem:mapsLR}, we have $\Lambda_{-\epsilon}(J) \subseteq I$ and $\Lambda_{-\epsilon}(I) \subseteq J$. Then we obtain $\Lambda_{-\epsilon}(I) \subseteq J \subseteq \Lambda_\epsilon(I)$, as desired.

As for the case of $\RR$, we can deduce the desired equivalences by a similar argument as the case of $\LL$. This completes the proof.
\end{proof}

\begin{lemma}\label{prop:interleaving-type}
Let $I$ and $J$ be intervals of $B$ and $\epsilon \geq0$.
    If $k_I$ and $k_J$ are $\Lambda_\epsilon$-interleaved, then $I$ and $J$ are of the same type, or $k_I$ and $k_J$ are $\Lambda_{2\epsilon}$-trivial. 
\end{lemma}

\begin{proof} Let the pair $\alpha \colon k_I \to k_J(\epsilon)$ and $\beta \colon k_J \to k_I(\epsilon)$ be a $\Lambda_\epsilon$-interleaving between $k_I$ and $k_J$.
If $I$ and $J$ are not the same type, then, we have  $0= \beta(\epsilon)\circ \alpha = (k_I)_{0\to 2\epsilon}$ and $0=\alpha(\epsilon)\circ \beta = (k_J)_{0\to 2\epsilon}$ by Corollary  \ref{rem:become-zero}. Thus, $k_I$ and $k_J$ are $\Lambda_{2\epsilon}$-trivial.
    \end{proof}

Before closing this subsection, we prove the implication {\rm(b)} $\Rightarrow$ {\rm(a)} in the statements of Theorem~\ref{thm:isometry}.

\begin{proof}[Proof of {\rm(b)} $\Rightarrow$ {\rm(a)} in Theorem \ref{thm:isometry}]
Suppose that $V$ and $W$ are bottleneck $\Lambda_\epsilon$-interleaved. We show  that $\mathscr{B}(V)$ and $\mathscr{B}(W)$ are $\Lambda_\epsilon$-matched.

 By Lemma \ref{prop:interleaving-type}, if $I \in \coim \sigma$ and $\sigma(I)$ are of different types, then they are $\Lambda_{\epsilon}$-trivial. By removing these intervals from $\coim \sigma$ and $\im \sigma$ respectively, we obtain a new bottleneck $\Lambda_\epsilon$-interleaving $\sigma' \colon \mathscr{B}(V) \nrightarrow \mathscr{B}(W)$ such that  $I \in \coim \sigma'$ and $\sigma'(I)$ are of the same type. 

     By Lemma \ref{lem:aaa}, if $I \in \coim \sigma'$ and $\sigma'(I)$ are $\Lambda_\epsilon$-interleaved, then both $I$ and $\sigma'(I)$ are $\Lambda_{2\epsilon}$-trivial, or $I \subseteq \Ex_{\epsilon}^{\Lambda} (\sigma'(I))$ and  $\sigma'(I) \subseteq \Ex_{\epsilon}^{\Lambda} (I)$ for all $I \in \coim \sigma'$. 
    By removing these $\Lambda_{2\epsilon}$-trivial intervals $I$ and $\sigma'(I)$ from $\coim \sigma'$ and $\im \sigma'$ respectively, we can construct a new bottleneck $\Lambda_\epsilon$-interleaving $\sigma''  \colon  \mathscr{B}(V) \nrightarrow \mathscr{B}(W)$ such that $I \subseteq \Ex_{\epsilon}^{\Lambda}(\sigma''(I))$ and $\sigma''(I) \subseteq  \Ex_{\epsilon}^{\Lambda}(I)$ for all $I \in \coim \sigma''$.  
    This is a $\Lambda_\epsilon$-matching between $\mathscr{B}(V)$ and $\mathscr{B}(W)$. 
\end{proof}

\subsection{Proof of the algebraic stability theorem of bipath persistence modules}\label{subsec:AST}
In this subsection, we prove the implication (c) $\Rightarrow$ (b) in Theorem \ref{thm:isometry}. For this purpose, we give the following lemma, where statements (1) and (2) are separated because (1) is the simpler case, while (2) contains the essential argument.

\begin{lemma}\label{lem:final}
     Let $V$ and $W$ be pfd bipath persistence modules.
\begin{enumerate}
    \item  [{\rm (1)}]  
For $\XX \in \{\UU, \DD, \BB \}$, if $V^{\XX}$ and $W^{\XX}$ are $\Lambda_\epsilon$-interleaved, then they are bottleneck $\Lambda_\epsilon$-interleaved. 
\item  [{\rm (2)}]  For $\XX \in \{\LL, \RR \}$, if $V^{\XX}$ and $W^{\XX}$ are $\Lambda_\epsilon$-interleaved, then they are bottleneck $\Lambda_\epsilon$-interleaved. 
\end{enumerate}
\end{lemma}

Using Lemma \ref{lem:final}, we prove the implication {\rm(c)} $\Rightarrow$ {\rm(b)} in the statements of Theorem \ref{thm:isometry}.
\begin{proof}[Proof of {\rm(c)} $\Rightarrow$ {\rm(b)} in Theorem \ref{thm:isometry}]\label{thm:AST}
Suppose that pfd bipath persistence modules $V$ and $W$ are $\Lambda_\epsilon$-interleaved.
   By Proposition \ref{lem:type-separating}, we have a $\Lambda_\epsilon$-interleaving between $V^{\XX}$ and $W^{\XX}$ for each $\XX \in\TT$. By Lemma~\ref{lem:final}, there exists a bottleneck $\Lambda_\epsilon$-interleaving $\sigma_{\XX}$ between $V^{\XX}$ and $W^{\XX}$ for each $\XX \in\TT$.  Then, we can construct a bottleneck $\Lambda_\epsilon$-interleaving $\sigma$ between $V^\XX$ and $W^\XX$ such that $\coim \sigma = \bigsqcup_{\XX \in\TT} \coim \sigma_{\XX}$ and $\im \sigma =  \bigsqcup_{\XX \in\TT} \im \sigma_{\XX}$ which sends $I \in \coim \sigma $ to $\sigma_\XX (I) \in \im \sigma_\XX$ for $I \in \coim \sigma_\XX$. 
\end{proof}

Below, we prove Lemma \ref{lem:final}. We begin with the proof of Lemma \ref{lem:final}(1).

\begin{proof}[Proof of Lemma \ref{lem:final}{\rm(1)}]
Suppose that $V^\XX$ and $W^\XX$ are $\Lambda_\epsilon$-interleaved for  $\XX \in \{\UU, \DD ,\BB\}$. 

As for the case $\XX =\BB$, we can check that a $\Lambda_\epsilon$-interleaving between $V^\BB$ and $W^\BB$ induces $V^\BB \cong W^\BB$. It implies that they are bottleneck $\Lambda_\epsilon$-interleaved.

As for the case $\XX =\UU$ or $\DD$, we first note that pfd bipath persistence module $V^\XX$ and  $W^\XX$ can be regarded as pfd $\mR$-persistence modules.
Then, we can apply the algebraic stability theorem of pfd $\mR$-persistence modules (e.g., \cite[Theorem 4.3]{bjerkevik2021onthesatbility}) and obtain a bottleneck $\Lambda_\epsilon$-interleaving between $V^\XX$ and $W^\XX$.
\end{proof}

 We next prove Lemma~\ref{lem:final}(2) using  a graph theory. In Lemma \ref{lem:final}(2), the case for $\RR$ can be proved by a similar discussion of $\LL$. Therefore, we consider the case $\XX=\LL$ below. The notation used below is given in Subsection~\ref{subsec:TheBipathPoset}.


For our purpose, we first define an addition on decorated numbers.  Let $s^\sigma, t^\tau \in \mR^*$. We define $s^\sigma + t^\tau$ by
\begin{equation*}
    s^\sigma + t^\tau :=
    \begin{cases}
        (s+t)^+ & \text{ if $(\sigma,\tau) =(+,+)$},\\
        (s+t)^- & \text{ else.}
    \end{cases}
\end{equation*}
It is easy to see that the addition satisfies the following properties. Let $s_1^{\sigma_1}, s_2^{\sigma_2},  t_1^{\tau_1},  t_2^{\tau_2},  \in \mR^*$.
\begin{enumerate}
    \item [{\rm (1)}] If $s_1^{\sigma_1} = s_2^{\sigma_2}$, then $s_1=s_2$. Moreover if $s_1^{\sigma_1} + s_2^{\sigma_2} = 
t_1^{\tau_1} +t_2^{\tau_2}$ then $s_1 +s_2 = t_1 +t_2$.  
 \item [{\rm (2)}] If $s_1^{\sigma_1} <_* s_2^{\sigma_2}$, then $s_1 \leq s_2$. Notice that the strict inequality $s_1 < s_2$ fails in general.
    \item [{\rm (3)}]  If $s_1^{\sigma_1} <_*  t_1^{\tau_1}$ and $s_2^{\sigma_2} \leq_*  t_2^{\tau_2}$, then we have $ s_1^{\sigma_1}+ s_2^{\sigma_2} \leq_* t_1^{\tau_1}+  t_2^{\tau_2}$. Notice that the strict inequality $ s_1^{\sigma_1}+ s_2^{\sigma_2} <_* t_1^{\tau_1}+  t_2^{\tau_2}$ fails in general.
\end{enumerate}

We next recall intervals in $\LL$. Any interval in $\LL$ is described by a pair of decorated numbers $\langle s^\sigma, t^\tau \rangle_\LL
$ with  $s^\sigma, t^\tau \notin \omR^* \setminus \{+\infty^+\}$. For our purpose, we set $\pm \infty + t :=\pm \infty$ and $ t \pm \infty := \pm \infty$ for $t\in \mR$.
\begin{lemma}\label{lem:add_dec}
     Let  $S =  \langle s_1^{\sigma_1}, s_2^{\sigma_2}  \rangle_\LL$ and $T = \langle  t_1^{\tau_1},t_2^{\tau_2} \rangle_\LL$ be intervals in $\LL$. We have the following statements.
    \begin{enumerate}
   \item [{\rm (1)}] We have $S \subseteq T$ if and only if both  $s_1^{\sigma_1} \leq_* t_1^{\tau_1}$ and $s_2^{\sigma_2} \leq_* t_2^{\tau_2}$ hold.
    \item [{\rm (2)}] For any $\epsilon \geq0$, we have  $\Lambda_{-\epsilon}(S) =  \langle  (s_1- \epsilon)^{\sigma_1},(s_2 -\epsilon)^{\sigma_2} \rangle_\LL$.
    \end{enumerate}
\end{lemma}
\begin{proof}
    (1) follows from $S,T \in \LL$. (2) is clear from the definition of the translation $\Lambda_\epsilon$.
\end{proof}

 We give a preorder on $\LL$.  Recall that $\LL$ is divided into ${\LL_i}$ $ (i=1,2,3,4,5)$. Using the division, we define a map $\ell  \colon \LL \to \overline{\mR}^* \colon$ 
 \begin{equation*}
     \ell (\langle s^{\sigma},t^{\tau}\rangle_\LL)\coloneq \begin{cases}
         - \infty^-  & \text{if } \langle s^{\sigma},t^{\tau}\rangle_\LL \in \LL_1, \\
         \max \{ s^\sigma, t^\tau \} & \text{if }  \langle s^{\sigma},t^{\tau}\rangle_\LL \in \LL_2 , \\
           s^\sigma + t^\tau & \text{if }  \langle s^{\sigma},t^{\tau}\rangle_\LL \in \LL_3 , \\
             \min \{ s^\sigma, t^\tau \} & \text{if }  \langle s^{\sigma},t^{\tau}\rangle_\LL \in \LL_4 , \\
         +\infty^- & \text{if }   \langle s^\sigma ,t^\tau \rangle_\LL \in \LL_5.
     \end{cases}
 \end{equation*} 
For intervals $I, J \in  \LL$, we write $I \leq_\ell J$ if  
\begin{itemize}
    \item $I \in \LL_i$ and $J \in \LL_j$ with $i<j$ for $i,j\in \{1,2,3,4,5\}$, or
    \item $I, J \in \LL_i$ for some $i\in \{1,2,3,4,5\}$ with $\ell(I) \leq_* \ell(J)$.  
\end{itemize}
Then, $\leq_\ell$ defines a preorder on $\LL$. Indeed, it
is equipped with reflexivity ($I \leq_\ell I$) and transitivity ($I \leq_\ell I'$ and  $I' \leq_\ell  I''$ implies $I \leq_\ell I''$). Also, we note that any intervals in $\LL$ are comparable by this preorder.

\begin{lemma}\label{lem:T<R<S<T}    Let $R$, $S$, and $T$ be intervals in $\LL$. We have the following:
\begin{enumerate}
    \item [{\rm (1)}]$R \subseteq S$ implies $R \leq_\ell S$.
    \item [{\rm (2)}] For any  $\epsilon \geq 0$ and $i\in \{1,2,3,4,5\}$, if $R$ in $\LL_i$, then $\Lambda_{-\epsilon}(R)\in \LL_i$.
    \item [{\rm (3)}] If we have $R \leq_\ell S$ and $S \leq_\ell T$ with $R, T \in \LL_i$ for some $i\in \{1,2,3,4,5\}$, then $S\in \LL_i$. 
\end{enumerate}
\end{lemma}
\begin{proof}
    We can obtain these claims immediately from the definitions of $\leq_\ell$ and $\LL_i$ for each $i=1,2,3,4$, and $5$.
\end{proof}

\begin{lemma}\label{lem:equiv}
For intervals $S, T \in \LL_i$ with $i\in \{ 2,4\}$, the following statements are equivalent.
\begin{enumerate}
    \item [\rm{(i)}] $S \subseteq T$.
    \item [\rm{(ii)}] $S \leq_\ell T$ and there exists $\epsilon \geq 0$ such that $\Lambda_{-\epsilon}(T) \subseteq S$.
\end{enumerate}
\end{lemma}

\begin{proof}
We prove the case for $i=2$; the case for $i=4$ is shown in a similar discussion.
Recall that intervals in $\LL_2$ are of the form  $\langle s^\sigma ,-\infty^+ \rangle_\LL$ or $\langle -\infty^+, t^\tau \rangle_\LL$ for some $s^\sigma, t^\tau \in \mR^*$. In addition, there are no inclusion relations between $\langle s^\sigma ,-\infty^+ \rangle_\LL$ and $\langle -\infty^+, t^\tau \rangle_\LL$.

Let $S, T \in \LL_2$.  We have the following equivalences. 
\begin{equation*}
    \begin{split}
    S \subseteq T. 
    &\iff \text{The pair }
    (S,T) \text{ is either } ( \langle   -\infty^+  , s^\sigma \rangle_\LL, \langle -\infty^+, t^\tau \rangle_\LL)  \text{ or } (\langle s^\sigma, -\infty^+ \rangle_\LL, \langle t^\tau,  -\infty^+ \rangle_\LL)  \\
    & \ \ \ \ \ \quad \text{ for some } s^\sigma \leq_* t^\tau \in \mR^*.   
\\
   &\iff \text{The pair }
    (S,T) \text{ is either } ( \langle   -\infty^+  , s^\sigma \rangle_\LL, \langle -\infty^+, t^\tau \rangle_\LL)  \text{ or } (\langle s^\sigma, -\infty^+ \rangle_\LL, \langle t^\tau,  -\infty^+ \rangle_\LL)  \\
    & \ \ \ \ \ \quad  \text{ for some } s^\sigma, t^\tau \in \mR^*  \text{ with } S\leq_\ell T. 
    \\
     & \iff 
     S \leq_\ell T \text{ and there exists } \epsilon \geq 0 \text{ such that } \Lambda_{-\epsilon}(T) \subseteq S.
    \end{split}
\end{equation*}
 We obtain the desired results. 
\end{proof}

The next observation is a key to Lemma~\ref{lem:final}(2). 

\begin{lemma}\label{lem:interleaved-RST}
Let $R$, $S$, and $T$ be intervals in $\LL$ such that
    $R \leq_\ell T $. If there exist nonzero morphisms $k_R \to k_{S}(\epsilon)$ and $k_S \to  k_T (\epsilon) $, then $k_S$ is $\Lambda_\epsilon$-interleaved with either $k_R$ or $k_T$. 
\end{lemma}

\begin{proof}Let $R$, $S$, and $T$ be intervals in $\LL$ such that
    $R \leq_\ell T$.
We first show that $R,S$, and $T$ belong to the same $\LL_i$ for some $i\in \{1,2,3,4,5\}$ as follows. By assumption, there exist nonzero morphisms $k_R \to k_{S}(\epsilon)$ and $k_S \to  k_T (\epsilon)$. Thus, by  Lemmas~ \ref{lem:support_interval} and \ref{lem:mapsLR}, we have 
\begin{equation}\label{eq:for-lemma}
    \Lambda_{-2\epsilon}(T) \subseteq \Lambda_{-\epsilon}(S)  \subseteq R.
\end{equation}
 In particular, by Lemma~\ref{lem:T<R<S<T}(1) and the assumption $R\leq_\ell T$, we have
\[
    \Lambda_{-2\epsilon}(T) \leq_\ell \Lambda_{-\epsilon}(S)  \leq_\ell R \leq_\ell T.
\]
Therefore, by Lemma~\ref{lem:T<R<S<T}(2)(3), if $T$ is in $\mathcal{L}_i$, then so are $R$ and $S$.

Using this observation, we prove the claim by considering each case where $T$ belongs to $\LL_i$ for $i\in \{1,2,3,4,5\}$. 

First, if $T$ is in $\LL_1$ or $\LL_5$, then we have $R= S=T$ since both $\LL_1$ and $\LL_5$ are singletons. Therefore, $k_S$ is $\Lambda_\epsilon$-interleaved with both $k_R$ and $k_T$.  

Next, let $T \in \LL_2$ (resp. $T \in \LL_4$). Now, we have $R \leq_\ell T$ by assumption and the inclusion relation $\Lambda_{-2\epsilon}(T) \subseteq R$ given in \eqref{eq:for-lemma}. Therefore, by Lemma~\ref{lem:equiv}, we have $R \subseteq T$. Combining it with \eqref{eq:for-lemma}, we have $\Lambda_{-\epsilon}(T) \subseteq S \subseteq \Lambda_\epsilon(T)$ and $\Lambda_{-\epsilon}(R) \subseteq S \subseteq \Lambda_\epsilon(R)$. By Lemma~\ref{lem:aaa}(2),  $k_S$ is $\Lambda_\epsilon$-interleaved with both $k_R$ and $k_T$.

Finally, we assume $T \in \LL_3$. We write $R$, $S$, and $T$ as $\langle r_1^{{\rho_1}} ,r_2^{{\rho_2}} \rangle_\LL$,   $\langle s_1^{{\sigma_1}} ,s_2^{{\sigma_2}} \rangle_\LL$, and  $\langle t_1^{{\tau_1}}  ,t_2^{{\tau_2}}  \rangle_\LL$ by pairs of decorated numbers, respectively. We show $k_S$ is $\Lambda_\epsilon$-interleaved with either $k_R$ or $k_T$ by contradiction.  

Assume that $k_S$ and $k_R$ are not $\Lambda_\epsilon$-interleaved, and  $k_S$ and $k_T$ are not $\Lambda_\epsilon$-interleaved. By Lemma~\ref{lem:aaa}(2), we have $\Lambda_{-\epsilon}(S) \nsubseteq R$ or $\Lambda_{-\epsilon}(R) \nsubseteq S$, and $\Lambda_{-\epsilon}(S) \nsubseteq T$ or $\Lambda_{-\epsilon}(T) \nsubseteq S$. 
By \eqref{eq:for-lemma}, we have $\Lambda_{-\epsilon}(R) \nsubseteq S$ and $\Lambda_{-\epsilon}(S) \nsubseteq T$. 
This implies that there exist $i$ and $p$ in $\{1,2\}$ such that 
\begin{equation}\label{eq:for-lemma2}
   s_i^{\sigma_i}   <_* (r_i-\epsilon)^{\rho_i} \text{ and } t_p^{\tau_p}  <_*  (s_p-\epsilon)^{\sigma_p},
\end{equation}
by Lemma \ref{lem:add_dec}(1).
 On the other hand, by $\Lambda_{-\epsilon}(T) \subseteq S$ and  $\Lambda_{-\epsilon}(S)  \subseteq R$ form \eqref{eq:for-lemma} and Lemma \ref{lem:add_dec}(2), we obtain 
\begin{equation}\label{eq:4eq}
\begin{split} 
(s_1 -\epsilon)^{\sigma_1} \leq_* r_1^{\rho_1} \text{ and }
     (s_2 -\epsilon)^{\sigma_2} \leq_* r_2^{\rho_2}, 
     \\
    (t_1-\epsilon)^{\tau_1} \leq_* s_1^{\sigma_1} \text{ and } (t_2-\epsilon)^{\tau_2} \leq_* s_2^{\sigma_2}.
\end{split}
\end{equation}
We take $j$ and $q$ from $\{1,2\}$ so that $i\neq j$ and $p\neq q$. For $j$ and $q$, we have
\begin{equation}\label{eq:for-lemma3}
    (s_j -\epsilon)^{\sigma_j} \leq_* r_j^{\rho_j} \text{ and } (t_q-\epsilon)^{\tau_q} \leq_* s_q^{\sigma_q},
\end{equation}
by \eqref{eq:4eq}.
In particular, we have 
\begin{equation}\label{eq:for-lemma5}
s_j-\epsilon \leq r_j \text{ and } t_q -\epsilon \leq s_q.    
\end{equation}

By \eqref{eq:for-lemma2} and \eqref{eq:for-lemma3}, we have 
\begin{equation*}
\begin{split}
s_1^{\sigma_1}  + s_2^{\sigma_2}  & \leq_* (r_i-\epsilon)^{\rho_i} +  (r_j + \epsilon)^{\rho_j} \\
&= \begin{cases}
    (r_1 + r_2)^+ & \text{if } (\rho_1,\rho_2)=(+,+) \\
    (r_1 + r_2)^- & \text{else,} 
\end{cases}  \\
& = r_1^{\rho_1}  + r_2^{\rho_2}.
\end{split}
\end{equation*}
Similarly, we obtain $t_1^{{\tau_1}} +t_2^{{\tau_2}} \leq_* s_1^{{\sigma_1}} + s_2^{{\sigma_2}}.$ 
On the other hand, we have $R \leq_\ell T$ by assumption. This is equivalent to $ r_1^{\rho_1}+ r_2^{\rho_2}  \leq_* t_1^{\tau_1}+ t_2^{\tau_2}$ by definition. Thus, we obtain 
\begin{equation}\label{eq:str*}
    s_1^{{\sigma_1}} + s_2^{{\sigma_2}} = r_1^{{\rho_1}} + r_2^{{\rho_2}} = t_1^{{\tau_2}} +t_2^{{\tau_2}}. 
\end{equation}
In particular, we have the equalities:
\begin{equation}\label{eq:triple_equality}
s_1 +s_2 =r_1 +r_2 =t_1 + t_2.    
\end{equation}

By $s_i^{\sigma_i}   <_* (r_i-\epsilon)^{\rho_i}$ from \eqref{eq:for-lemma2}, we have either $s_i < r_i- \epsilon $, or $ s_i= r_i- \epsilon $ with $  (\sigma_i,\rho_i) =(-,+)$. If the former inequality $s_i < r_i- \epsilon$ holds, then using $s_j-\epsilon \leq r_j$ from \eqref{eq:for-lemma5}, we obtain  $s_1 + s_2 < r_1 + r_2 $. This contradicts to \eqref{eq:triple_equality}. Therefore, we have $s_i = r_i- \epsilon $ with  $(\sigma_i,\rho_i) =(-,+)$. 
Then, we have $s_j -\epsilon=r_1 + r_2 - (s_i + \epsilon) = r_j$  and ${\rho_j} = -$ by \eqref{eq:str*}.
By substituting each for $(s_j -\epsilon)^{\sigma_j} \leq_* r_j^{\rho_j}$, we obtain $r_j^{\sigma_j} \leq_* r_j^{-}$,
 which implies ${\sigma_j} =-$. We note that $(\sigma_1, \sigma_2)=(-,-)$ holds.
By $t_p^{\tau_p}  <_*  (s_p-\epsilon)^{\sigma_p}$ from \eqref{eq:for-lemma2}, we have either $t_p < s_p- \epsilon$, or $t_p = s_p- \epsilon $ with $(\tau_p, \sigma_p)=(-,+)$. 
Since we have $(\sigma_1,\sigma_2)=(-,-)$, the latter fails, so we must have $t_p < s_p- \epsilon$. In this case, using $t_q -\epsilon \leq s_q$ from \eqref{eq:for-lemma5}, we obtain $t_1 +t_2 < s_1 +s_2$, which contradicts to \eqref{eq:triple_equality}. Thus, the former case also fails.  This is a contradiction. Therefore $k_S$ is $\Lambda_\epsilon$-interleaved with either $k_R$ or $k_T$.  This completes the proof.
\end{proof}

Using the above, we prove the following statement, which gives a proof of Lemma~\ref{lem:final}{\rm(2)} for the case of $\LL$. 
 \begin{proposition}\label{prop:hall}
 Let $V$ and $W$ be pfd bipath persistence modules. For $\epsilon\geq 0$, we have the following statements.
 \begin{enumerate}
     \item [{\rm(1)}]   $\mathscr{B}_{2\epsilon}(V^\LL)= \mathscr{B}_{}(V^\LL)$ and $\mathscr{B}_{2\epsilon}(W^\LL)= \mathscr{B}_{}(W^\LL)$.
        \item [{\rm(2)}]  
        Let $G=(\mathscr{B}(V^\LL),\mathscr{B}(W^\LL);E_{\mu_\epsilon^\Lambda})$ be the bipartite graph.
        If $V^\LL$ and $W^\LL$ are $\Lambda_\epsilon$-interleaved, then the full subgraphs $G(\mathscr{B}_{2\epsilon}(V^\LL), \mathscr{B}(W^\LL))$ and $G(\mathscr{B}(V^\LL), \mathscr{B}_{2\epsilon}(W^\LL))$ satisfy Conditions {\rm (H)} and  {\rm (H$'$)} (Definition \ref{condition}), respectively. 
 \end{enumerate}
 Consequently,
if $V^\LL$ and $W^\LL$ are $\Lambda_\epsilon$-interleaved, then there exists a bottleneck $\Lambda_\epsilon$-interleaving between $V^\LL$ and $W^\LL$.
\end{proposition}

\begin{proof}
The assertion (1) follows from Lemma~\ref{lem:aaa}(2).


We prove (2). 
We first note that we have $G(\mathscr{B}_{2\epsilon}(V^\LL), \mathscr{B}(W^\LL)) =G=G(\mathscr{B}(V^\LL), \mathscr{B}_{2\epsilon}(W^\LL))$ by (1). In addition, both $\mathscr{B}(V^\LL)$ and $\mathscr{B}(W^\LL)$ are finite since $V^\LL$ and $W^\LL$ are pfd bipath persistence modules and since any interval in both sets contains $-\infty\in B$.

 We show that the bipartite graph $G$ satisfies Condition (H), i.e., $|X|\leq |\mu_\epsilon^\Lambda (X)|$ for any $X\subseteq \mathscr{B}(V^\LL)$. 
 The assertion for (H$'$) can be obtained by a similar discussion. 

 Let $X$ be a subset of $\mathscr{B}(V^\LL)$. Let $n:= |X|$ and $m:= |\mu_\epsilon^\Lambda(X)|$, respectively.
Since any intervals in $\LL$ are comparable with respect to $\leq_\ell$ as mentioned below,  we order $X=\{I_1, \hdots , I_n\}$ so that $I_i \leq_\ell I_{i'}$ for $1\leq i \leq i' \leq n$. We let $\mu_\epsilon^\Lambda (X) = \{J_1,\hdots, J_m\}$.

Let $\alpha\colon V^\LL \to W^\LL(\epsilon)$ and $\beta \colon W^\LL \to V^\LL(\epsilon)$ be a $\Lambda_\epsilon$-interleaving between $V^\LL$ and $W^\LL$. We write $\alpha$ and $\beta$ by
 \[\alpha =(\alpha_{J,I})_{I \in \mathscr{B}(V^\LL), J \in \mathscr{B}(W^\LL)} \colon V^\LL \to  W^\LL(\epsilon) \text{ and } \beta =(\beta_{I,J})_{I \in \mathscr{B}(V^\LL), J \in \mathscr{B}(W^\LL)} \colon W^\LL\to V^\LL(\epsilon).\]
 Then, for any $I_i,I_{i'} \in X$ with $i\leq i'$ and $J \in \mathscr{B}(W^\LL) \setminus \mu_\epsilon^\Lambda (X) $, we obtain
 \begin{equation}\label{eq:makezero}
\beta_{I_{i'},J} (\epsilon) \circ \alpha_{J,I_i}  = 0    
 \end{equation}
by the following discussion. Suppose $\beta_{I_{i'},J} (\epsilon) \circ \alpha_{J,I_i}  \neq 0$. Then,  we have nonzero morphisms $\alpha_{J,I_i}\colon k_{I_i} \to k_J(\epsilon) $ and $\beta_{I_{i'},J} \colon k_J \to k_{I_{i'}}(\epsilon)$. Thus, by Lemma \ref{lem:interleaved-RST}, $k_J$ is $\Lambda_\epsilon$-interleaved with either $k_{I_i}$ or $k_{I_{i'}}$. This implies $J\in \mu_\epsilon^\Lambda (X)$, which  contradicts to the assumption $J\in \mathscr{B}(W^\LL)\setminus  \mu_\epsilon^\Lambda (X)$. 

Using  \eqref{eq:makezero}, for any $I_i\in X$, we have

\begin{equation}\label{eq:11111}
    \begin{split}
         (k_{I_i})_{0\to 2\epsilon} &=  \sum_{J \in \mathscr{B}(W^\LL)} \beta_{{I_i},J}(\epsilon) \circ      \alpha_{J_j,{I_i}} \\
    &=  \sum_{j =1 }^m \beta_{{I_i},J_j}(\epsilon) \circ      \alpha_{J_j,{I_i}} 
    +
    \sum_{J \in \mathscr{B}(W^\LL) \setminus \mu_\epsilon^\Lambda(X)} \beta_{{I_i},J}(\epsilon) \circ      \alpha_{J,{I_i}}\\
    &\overset{\eqref{eq:makezero}}{=}
    \sum_{j =1 }^m \beta_{{I_i},J_j}(\epsilon) \circ  \alpha_{J_j,{I_i}}.
    \end{split}
\end{equation}
Similarly, for $I_i, I_{i'}  \in X$ with $i < i'$, we have
\begin{equation}\label{eq:222222}
    \begin{split}
       0 & =  \sum_{J \in \mathscr{B}(W^\LL)} \beta_{I_{i'},J}(\epsilon) \circ      \alpha_{J,I_i} \\
    & =   \sum_{j=1}^m \beta_{I_{i'},J_j}(\epsilon) \circ      \alpha_{J_j,I_i} 
    +
    \sum_{J \in \mathscr{B}(W^\LL) \setminus \mu_\epsilon^\Lambda(X)} \beta_{I_{i'},J}(\epsilon) \circ      \alpha_{J,I_i}\\
     &\overset{\eqref{eq:makezero}}{=}
     \sum_{j=1}^m \beta_{I_{i'},J_j}(\epsilon) \circ      \alpha_{J_j,I_i}.      
    \end{split}
\end{equation}

 \def\objectstyle{\displaystyle}

We denote by $\psi$ the composition of $(\alpha_{J,I})_{J \in \mu_\epsilon^\Lambda(X), I \in X }$ and $(\beta_{I, J}(\epsilon))_{J \in \mu_\epsilon^\Lambda(X), I \in X }$, which is displayed by 
\begin{equation*}
\xymatrix{
(\bigoplus_{i =1}^{n} k_{I_i}) \ar@{->}[rr]^{\psi} \ar@{->}[rd] _{(\alpha_{J_j,I_i})_{J_j \in \mu_\epsilon^\Lambda(X), I_i \in X } }
&     \ar@{}[d]|{\circlearrowright} 
& (\bigoplus_{i =1}^{n} k_{I_i}(2\epsilon))  \\ 
& (\bigoplus_{j=1}^{m} k_{J_j} (\epsilon))  \ar@{->}[ru]_(.5){\quad (\beta_{I_i, J_j}(\epsilon))_{J_j \in \mu_\epsilon^\Lambda(X), I_i \in X }}.
& 
}  
\end{equation*}
By equations \eqref{eq:11111} and \eqref{eq:222222}, the morphism  $\psi$ is of the form with respect to the ordering $X=\{I_1,\hdots I_n\}$:
\begin{equation*}
 \psi=   \begin{bmatrix}
   (k_{I_1})_{0 \to 2\epsilon } & \ast & \cdots & \ast  \\
  0 &  (k_{I_2})_{0 \to 2\epsilon } & \cdots & \ast \\
  \vdots & \vdots &  \ddots &\vdots \\
  0 & 0 & \cdots &  (k_{I_n})_{0 \to 2\epsilon}
\end{bmatrix}.
\end{equation*}
Since every interval in $\LL$ contains the minimal element $-\infty \in B$,  the $k$-linear map 
\[
{\psi}_{-\infty} \colon (\bigoplus_{i =1}^{n} {k_{I_i}})_{-\infty }(\cong k^n) \to (\bigoplus_{i =1}^{n} k_{I_i}(2\epsilon))_{-\infty } (\cong k^n)\]
is an isomorphism as vector space. Thus, the $k$-linear map 
\[((\alpha_{J,I})_{J \in \mu_\epsilon^\Lambda(X), I \in X })_{-\infty} \colon (\bigoplus_{i= 1}^{n} {k_{I_i}})_{-\infty } (\cong k^{n}) \to   (\bigoplus_{j=1 }^{m} k_{J_j}(\epsilon))_{-\infty} (\cong k^{m}) \]
is injective. Therefore, we obtain the desired inequality \[ |X| = n\leq m  =|\mu_\epsilon^\Lambda(X)|. \]

The last claim is immediately followed by (2) and Proposition \ref{thm:matching}.
\end{proof}

We prove Lemma \ref{lem:final}(2) and conclude the assertion of  Lemma \ref{lem:final}.

\begin{proof}[Proof of Lemma \ref{lem:final}{\rm(2)}] 
As mentioned before, we only consider the case of $\LL$. In this case, by Proposition \ref{prop:hall}, we obtain the desired result. 
\end{proof}

\subsection{Proof of the isometry theorem of bipath persistence modules}\label{subsec:isometry_bipath} 
Finally, we prove Theorem \ref{thm:isometry}.

\begin{proof}[Proof of Theorem~\ref{thm:isometry}] Let $V$ and $W$ be pfd bipath persistence modules.  Recall that Theorem \ref{thm:isometry} claims that the following are equivalent. \begin{enumerate}
        \item[ {\rm (a)} ] There exists a $\Lambda_\epsilon$-matching between $\mathscr{B}(V)$ and $\mathscr{B}(W)$.
        \item[ {\rm (b)} ] There exists a bottleneck $\Lambda_\epsilon$-interleaving between $V$ and $W$.
        \item[  {\rm (c)} ] There exists a $\Lambda_\epsilon$-interleaving between $V$ and $W$.
    \end{enumerate}
By Proposition \ref{prop:dI<dB}, we have (b) $\Rightarrow$ (c).
By Proposition~\ref{lem:inequality}, we have $\text{(a) $\Rightarrow$ (b)}$. In addition, the implications (b) $\Rightarrow$ (a) and (c) $\Rightarrow$ (b) have been proven in  Subsection~\ref{subsec:mor_interval} and Subsection~\ref{subsec:AST}, respectively.
 This completes the proof.
\end{proof}